\documentclass[11pt]{amsart}

\usepackage{a4wide}
\usepackage{amssymb}
\usepackage{mathrsfs}
\usepackage{enumerate}
\usepackage{esint}
\usepackage{titletoc}
\usepackage[colorlinks=true, urlcolor=blue, linkcolor=blue, citecolor=black]{hyperref}
\usepackage[nameinlink]{cleveref}
\usepackage{graphicx}

\theoremstyle{plain}
\newtheorem{theorem}{Theorem}[section]
\newtheorem{lemma}[theorem]{Lemma}
\newtheorem{corollary}[theorem]{Corollary}
\newtheorem{proposition}[theorem]{Proposition}
\newtheorem{assumption}{Assumption}

\theoremstyle{definition}
\newtheorem{definition}[theorem]{Definition}
\newtheorem{remark}[theorem]{Remark}

\numberwithin{equation}{section}

\renewcommand{\d}{\textnormal{d}}
\newcommand{\N}{\mathbb{N}}
\newcommand{\R}{\mathbb{R}}

\newcommand{\ma}{\mu_{\mathrm{axes}}}
\DeclareMathOperator{\supp}{supp}

\newcommand{\cE}{\mathcal{E}}
\newcommand{\cK}{\mathcal{K}}
\DeclareMathOperator*{\BMO}{BMO}
\renewcommand{\div}{\textnormal{div}}

\newcommand{\BIGOP}[1]
{
\mathop{\mathchoice%
{\raise-0.22em\hbox{\huge $#1$}}%
{\raise-0.05em\hbox{\Large $#1$}}{\hbox{\large $#1$}}{#1}}}
\newcommand{\bigtimes}{\BIGOP{\times}}

\makeatletter
\@namedef{subjclassname@2020}{%
  \textup{2020} Mathematics Subject Classification}
\makeatother

\title[Fractional orthotropic $p$-Laplacians]{Regularity estimates for fractional orthotropic $p$-Laplacians of mixed order}

\author{Jamil Chaker}
\address{Fakult\"at f\"ur Mathematik, Universit\"at Bielefeld, 33615 Bielefeld, Germany}
\email{jchaker@math.uni-bielefeld.de}

\author{Minhyun Kim}
\address{Fakult\"at f\"ur Mathematik, Universit\"at Bielefeld, 33615 Bielefeld, Germany}
\email{minhyun.kim@uni-bielefeld.de}

\thanks{Jamil Chaker gratefully acknowledges support by the Deutsche Forschungsgemeinschaft (SFB 1283/2 2021 – 317210226). Minhyun Kim gratefully acknowledges funding by the Deutsche Forschungsgemeinschaft (GRK 2235/2 2021 - 282638148).}

\begin{document}

\begin{abstract}
We study robust regularity estimates for a class of nonlinear integro-differential operators with anisotropic and singular kernels. In this paper, we prove a Sobolev-type inequality, a weak Harnack inequality, and a local H\"older estimate.
\end{abstract}

\maketitle


\section{Introduction} \label{sec:introduction}
In this paper we investigate regularity estimates for weak solutions to nonlocal equations 
\begin{equation}\label{eq:PDE}
 Lu=f \quad \text{in } Q=(-1,1)^n,
\end{equation}
where $L$ is a nonlinear integro-differential operator of the form
\begin{equation}\label{def:nonlocaloperator}
Lu(x) = \mathrm{PV} \int_{\mathbb{R}^n} |u(y) - u(x)|^{p-2} (u(y)-u(x)) \mu(x, \mathrm{d}y)
\end{equation}
for $p > 1$ and $f\in L^q(Q)$ for some sufficiently large $q$. 
The operator $L$ is clearly determined by the family of measures $(\mu(x,\d y))_{x\in\R^n}$. In the special case, when $L$ is the generator of a L\'evy process, $\mu(x,A)$ measures the number of expected jumps from $x$ into the set $A$ within the unit time interval. However, the class of operators that we consider in this paper is more involved and for that reason we first take a look at an important example.
Let $n\in\N$. For $s_1, \cdots, s_n \in (0,1)$, we define
\begin{equation*}
\mu_{\mathrm{axes}}(x,\mathrm{d}y) = \sum_{k=1}^n s_k(1-s_k) |x_k-y_k|^{-1-s_k p} \mathrm{d}y_k \prod_{i\neq k} \delta_{x_i}(\mathrm{d}y_i).
\end{equation*}
This family plays a central role in our paper, since admissible operators resp. families of measures will be defined on the basis of $\ma$.
Given $x\in\R^n$, the measure $\ma(x,\cdot)$ only charges differences that occur along the axes
\[ \{x+te_k \, | \, t\in\R\} \quad \text{for } k\in\{1,\dots,n\}. \]
Hence, we can think of the operator $Lu$ for $\mu(x,\cdot)=\ma(x,\cdot)$
as the sum of one-dimensional fractional $p$-Laplacian in $\R^n$ with orders of differentiability $s_1,\dots,s_n\in(0,1)$ depending on the respective direction.
In particular $\ma(x,\cdot)$ does not possess a density with respect to the Lebesgue measure.
An interesting phenomenon for the case $p=2$ and $s=s_1=\dots=s_n$ is that on one hand the corresponding energies for the fractional Laplacian and $L$ are comparable. On the other hand (for sufficiently good functions) both operators converge to the Laplace operator as $s\nearrow 1$. It is known that the fractional $p$-Laplacian converges to the $p$-Laplacian (see \cite[Theorem 2.8]{BucSqu21} or \cite[Lemma 5.1]{dTGCV20} for details), that is  defined by
\[ \Delta_pu(x) = \div\left(|\nabla u(x)|^{p-2}\nabla u(x)\right). \]
However, the operator $L$ for $\mu(x,\cdot)=\ma(x,\cdot)$ converges for any $p>1$ and $s=s_1=\dots=s_n$ to the following local operator (up to a constant depending on $p$ only)
\begin{equation} \label{eq:Aploc}
  A_{\text{loc}}^pu(x)=\sum_{i=1}^n \frac{\partial}{\partial x_i} \left(\left|\frac{\partial u(x)}{\partial x_i}\right|^{p-2}\frac{\partial u(x)}{\partial x_i}\right) = \div\left(a\left(\nabla u(x)\right)\right)
\end{equation}
 as $s\nearrow 1$, where $a:\R^n\to\R^n$ with $a(z) = (|z_i|^{p-2}z_i)_{i\in\{1,\dots,n\}}$. This convergence is a direct consequence of the convergence for the one-dimensional fractional $p$-Laplacian and the summation structure of the operator for $\ma$. For details, we refer the reader to \Cref{prop:convergence}.
The operator $A_{\text{loc}}^p$ is known as orthotropic $p$-Laplacian and is a well-known operator in analysis (see for instance \cite[Chapter 1, Section 8]{Lions69}). This operator is sometimes also called pseudo $p$-Laplacian.
Minimizers for the corresponding energies have been studied in \cite{BEKA04}, where the authors prove for instance H\"older continuity of minimizers. In \cite{BBLV18},  local Lipschitz regularity for weak solutions to orthotropic $p$-Laplace equations for $p\geq 2$ and every dimension is proved. The case, when $p$ is allowed to be different in each direction, is also studied in several papers.
For instance in \cite{PalaPseudo}, the authors introduce anisotropic De Giorgi classes and study related problems. Another interesting paper studying such operators with nonstandard growth condition is \cite{BB20}, where the authors show that bounded local minimizers are locally Lipschitz continuous. For further results, we refer the reader to the references given in the previously mentioned papers.

The two local operators $\Delta_p$ and $A_{\text{loc}}^p$ are substantially different, as for instance $\Delta_p$ is invariant under orthogonal transformation, while $A_{\text{loc}}^p$ is not. One strength of our results is that they are robust and we can recover results for the orthotropic $p$-Laplacian by taking the limit. 

One way to deal with the anisotropy of $\ma$, is to consider for given 
$s_1,\dots,s_n\in(0,1)$ a class of suitable rectangles instead of cubes or balls. For this purpose we define $s_{\max} = \max\lbrace s_1, \cdots, s_n \rbrace$.
\begin{definition}\label{def:M_r}
For $r>0$ and $x\in\R^n$ we define 
\begin{align*}
M_r(x) =\bigtimes_{k=1}^n 
\left(x_k-r^{\frac{s_{\max}}{s_k}},x_k+r^{\frac{s_{\max}}{s_k}}\right) 
\quad \text{ and } M_r = M_r(0) \,.
\end{align*}
\end{definition}
The advantage of taking these cubes is that they take the anisotropy of the measures resp. operators into account and the underlying metric measure space is a doubling space. The choice of $s_{\max}$ in the definition of $M_r(x)$ is not important. It can be replaced by any positive number $\varsigma \geq s_{\max}$. We only need to ensure that $M_r(x)$ are balls in a metric measure space with radius $r > 0$ and center $x \in \mathbb{R}^n$. This allows us to use known results on doubling spaces like the John--Nirenberg inequality or results on the Hardy--Littlewood maximal function. 

In the spirit of \cite{CKW19}, for each $k \in \lbrace 1,\dots,n\rbrace$, we define $E_r^k(x) = \lbrace y \in \R^n : \vert x_k - y_k \vert < r^{s_{\max}/{s_k}}\rbrace$. Note, that
\begin{align}
\label{def:E_r}
M_r(x) = \bigcap_{k = 1}^n E_r^k(x).
\end{align}
We consider families of measures $\mu(x,\d y)$ which are given through certain properties regarding the reference family $\ma(x,\d y)$.
Let us introduce and briefly discuss our assumptions on the families $(\mu(x,\cdot))_{x\in\R^n}$.
\begin{assumption}\label{assumption:symmetry}
We assume
\begin{equation*}
\sup_{x\in\R^n} \int_{\mathbb{R}^n} (|x-y|^p \land 1) \mu(x,\mathrm{d}y) < \infty
\end{equation*}
and for all sets $A,B \in \mathcal{B}(\R^n)$:
\begin{align*}
\int_A \int_B \mu(x,\d y) \d x = \int_B \int_A \mu(x,\d y) \d x.
\end{align*}
\end{assumption}
\Cref{assumption:symmetry} provides integrability and symmetry of the family of measures. 
Furthermore, we assume the following tail behavior of $(\mu(x,\cdot))_{x\in\R^n}$.
\begin{assumption}\label{assumption:tail}
There is $\Lambda\geq 1$ such that for every $x_0 \in \R^n$, $k \in \lbrace 1, \dots , n \rbrace$ and all $r>0$
\begin{align*}
\mu(x_0, \R^n \setminus E_{r}^k(x_0)) \le 
\Lambda (1-s_k)r^{-ps_{\max}}.
\end{align*}
\end{assumption}
Note that \Cref{assumption:tail} is a stronger assumption than an assumption on the volume on the complement of every $M_r(x_0)$. It gives an appropriate tail behavior for the family of measures in each direction separately and allows us to control the appearing constants in our tail estimate in all directions. This is necessary to prove robust estimates for the corresponding operators. \\
Note that by \Cref{assumption:tail} and \eqref{def:E_r}, we have 
\begin{align}
\label{assmu1}
 \mu(x_0, \R^n \setminus M_{\rho}(x_0)) \le \sum_{k = 1}^n \mu(x_0, \R^n \setminus E_{\rho}^k(x_0)) \le \Lambda \sum_{k = 1}^n (1-s_k) \rho^{-ps_{\max}} \leq \Lambda n\rho^{-ps_{\max}}. 
 \end{align}
Hence, \eqref{assmu1} shows that \Cref{assumption:tail} implies $\mu(x_0,\R^n\setminus M_{\rho}(x_0)) \leq c\ma(x_0,\R^n\setminus M_{\rho}(x_0))$ for all $x_0\in\R^n$. \\
Finally, we assume local comparability of corresponding functionals. 
For this purpose, we define for any open and bounded $\Omega\subset \R^n$ 
 \begin{equation*}
\mathcal{E}_\Omega^{\mu}(u,v) = \int_\Omega \int_\Omega |u(y) - u(x)|^{p-2}(u(y)-u(x))(v(y)-v(x)) \mu(x, \mathrm{d}y) \mathrm{d}x
\end{equation*}
and $\cE^{\mu}(u,v)=\cE_{\R^n}^{\mu}(u,v)$ whenever these quantities are finite.
\begin{assumption}\label{assumption:comparability}
There is $\Lambda\geq 1$ such that for every $x_0 \in \R^n$, $\rho \in (0,3)$ and every \\
$u \in L^p(M_{\rho}(x_0))$:
\begin{align}
\label{assmu3}
& \Lambda^{-1} \cE^{\mu}_{M_{\rho}(x_0)}(u,u) \le \cE^{\ma}_{M_{\rho}(x_0)}(u,u) \le \Lambda \cE^{\mu}_{M_{\rho}(x_0)}(u,u).
\end{align}
\end{assumption}
Local comparability of the functionals is an essential assumption on the family of measures. It tells us that our family of measures can vary from our reference family in the given sense of local functionals without losing
crucial information on $(\mu(x,\cdot))_{x\in\R^n}$ like functional inequalities, which we deduce for the explicitly known family $(\ma(x,\cdot))_{x\in\R^n}$. This assumption allows us for instance to study operators of the form \eqref{def:nonlocaloperator} for $\ma$ in the general framework of bounded and measurable coefficients. We emphasize that further examples of families of measures satisfying \eqref{assmu3} can be constructed similarly to the case $p=2$ (see \cite[Section 9]{CKW19}).\\
In this paper, we study nonlocal operators of the form \eqref{def:nonlocaloperator} for families of measures that satisfy the previously given assumptions.
\begin{definition}\label{def:admissible}
Let $p> 1$, $\Lambda \ge 1$, and $s_1,\dots,s_n\in[s_0,1)$ be given for some $s_0\in(0,1)$.
We call a family of measures $(\mu(x,\cdot))_{x\in\R^n}$ admissible with regard to $(\ma(x,\cdot))_{x\in\R^n}$,
if it satisfies \Cref{assumption:symmetry}, \Cref{assumption:tail}, and \Cref{assumption:comparability}. We denote the class of such measures by $\cK(p,s_0,\Lambda)$.
\end{definition}
It is not hard to see that the family $(\ma(x,\cdot))_{x\in\R^n}$ is admissible in the above sense. Note that \Cref{assumption:symmetry} and \Cref{assumption:comparability} are clearly satisfied. Furthermore, for every $x_0 \in \R^n$, $k \in \lbrace 1, \dots , n \rbrace$ and all $r>0$
\[ \ma(x_0, \R^n \setminus E_{r}^k(x_0)) = 2s_k(1-s_k) \int_{r^{s_{\max}/s_k}}^{\infty} h^{-1-s_kp} = \frac{2(1-s_k)}{p}r^{-s_{\max}p}, \]
which shows \Cref{assumption:tail} for $\Lambda=\frac{2}{p}.$

The purpose of this paper is to study weak solutions to nonlocal equations governed by the class of operators $L$ as in \eqref{def:nonlocaloperator}. In order to study weak solutions, we need appropriate Sobolev-type function spaces which guarantee regularity and integrability with respect to $\mu$.
\begin{definition}\label{VHomega}
Let $\Omega\subset\R^n$ open and $p> 1$. We define the function spaces
\begin{align*}  
  V^{p,\mu}(\Omega|\R^n)  &= \Big\{ u: \,\R^n\to\R \text{ meas.} \, | \, u\bigr|_{\Omega}\in 
L^p(\Omega), (u,u)_{V^{p,\mu}(\Omega|\R^n)} <\infty\Big\}\,, 
 \\
 H^{p,\mu}_{\Omega}(\R^n) &= \Big\{ u: \,\R^n\to\R \text{ meas.}  \, | \, u\equiv 0 \text{ on } 
\R^n\setminus\Omega, \|u\|_{H^{p,\mu}_{\Omega}(\R^n)}<\infty \Big\}, 
\end{align*}
where
\begin{align*}
(u,v)_{V^{p,\mu}(\Omega|\R^n)} &= \int_{\Omega}\int_{\R^n} 
|u(y) - u(x)|^{p-2}(u(y)-u(x))(v(y)-v(x))\, \mu(x,\d y)\, \d x \,, \\
\|u\|_{H^{p,\mu}_{\Omega}(\R^n)}^p &= \|u\|_{L^p(\Omega)}^p + 
\int_{\R^n}\int_{\R^n} |u(y)-u(x)|^p\mu(x,\d y)\,\d x \,.
\end{align*}
\end{definition}
The space $V^{p,\mu}(\Omega|\R^n)$ can be seen as a nonlocal analog of the space $H^{1,p}(\Omega)$. 
It provides fractional regularity (measured in terms of $\mu$) inside of $\Omega$ 
and integrability on $\R^n \setminus \Omega$. The space $V^{p,\mu}(\Omega|\R^n)$ will serve as solution space.
On the other hand, the space $H^{p,\mu}_{\Omega}(\R^n)$ can be seen as a nonlocal analog of 
$H^{1,p}_0(\Omega)$. See \cite{FKV15} and \cite{DyKa17} for further studies of 
these spaces in the case $p=2$. 

We are interested in finding robust regularity estimates for weak solutions to a class of nonlocal equations. This means that the constants in the regularity estimates do not depend on the orders of differentiability of the integro-differential operator itself but only on a lower bound of the orders. Let us formulate the main results of this paper. For this purpose we define $\bar{s}$ to be the harmonic mean of the orders $s_1,\dots,s_n$, that is
\begin{equation*}
 \bar{s} = \left(\frac{1}{n} \sum_{k=1}^n \frac{1}{s_k}\right)^{-1}.
\end{equation*}
It is well known that the Harnack inequality fails for weak solutions to singular equations of the type \eqref{eq:PDE}. 
Even in the case $p=2$ and $s_1=\dots=s_n$, the Harnack inequality does not hold (See for instance \cite{BoSz07, BaCh10}). Our first main result is a weak Harnack inequality for weak supersolutions to equations of the type \eqref{eq:PDE}. Throughout the paper, we denote by $p_{\star}=np/(n-p\bar{s})$ the Sobolev exponent, which will appear in \Cref{thm:sobolev}.

\begin{theorem}[Weak Harnack inequality]\label{thm:weak_Harnack}
Let $\Lambda \ge 1$ and $s_1,\dots,s_n\in[s_0,1)$ be given for some $s_0\in(0,1)$. Let $1<p< n/\bar{s}$ and $f\in L^{q/(p\bar{s})}(M_1)$ for some $q>n$.
There are $p_0 = p_0(n,p,p_{\star},s_0,q,\Lambda)\in(0,1)$ and $C = C(n,p,p_{\star},s_0,q,\Lambda) > 0$ such that for each $\mu\in\cK(p,s_0,\Lambda)$ and every $u\in V^{p,\mu}(M_1|\R^n)$ satisfying $u\geq 0$ in $M_1$ and
\[ \cE^{\mu}(u,\varphi) \geq  (f,\varphi) \quad \text{for every non-negative } \varphi\in H^{p,\mu}_{M_1}(\R^n),\]
the following holds:
\begin{equation}\label{eq:thm:weakHarnack}
\begin{aligned}
\inf_{M_{1/4}} u \geq C \left( \fint_{M_{1/2}} u^{p_0}(x) \,\mathrm{d}x \right)^{1/p_0} &- \sup_{x \in M_{15/16}} 2 \left( \int_{\mathbb{R}^n \setminus M_1} u^-(z)^{p-1} \mu(x, \mathrm{d}z) \right)^{1/(p-1)} \\
& - \|f\|_{L^{q/(p\bar{s})}(M_{15/16})}.
\end{aligned}
\end{equation}
\end{theorem}
Although the weak Harnack inequality provides an estimate on the infimum only, it is sufficient to prove a decay of oscillation for bounded weak solutions and therefore a local H\"older estimate.
\begin{theorem}[Local H\"older estimate]\label{thm:Holder}
Let $\Lambda \ge 1$ and $s_1,\dots,s_n\in[s_0,1)$ be given for some $s_0\in(0,1)$. Let $1<p< n/\bar{s}$ and $f\in L^{q/(p\bar{s})}(M_1)$ for some $q>n$.
There are $\alpha = \alpha(n,p,p_{\star},s_0,q,\Lambda)\in(0,1)$ and $C = C(n,p,p_{\star},s_0,q,\Lambda) > 0$ such that for each $\mu\in\cK(p,s_0,\Lambda)$ and every $u\in V^{p,\mu}(M_1|\R^n)\cap L^{\infty}(\R^n)$ satisfying
\[ \cE^{\mu}(u,\varphi) =  (f,\varphi) \quad \text{for every } \varphi\in H^{p,\mu}_{M_1}(\R^n),\]
the following holds: $u\in C^{\alpha}(\overline{M}_{1/2})$ and 
\begin{equation*}
\|u\|_{C^\alpha(\overline{M}_{1/2})} \leq C\left( \|u\|_{L^{\infty}(\R^n)} + \|f\|_{L^{q/(p\bar{s})}(M_{15/16})} \right).
\end{equation*}
\end{theorem}
Note that the result needs global boundedness of weak solutions. The same assumption is also needed in the previous works \cite{CKW19, DK20, CK20}.
Global and local boundedness of weak solutions to anisotropic nonlocal equations are nontrivial open questions. 

Furthermore, note that the general case, replacing $M_{\frac{1}{4}}, M_{\frac{1}{2}}, M_{\frac{15}{16}}, M_1$ by $M_{\frac{r}{4}}(x_0)$, $M_{\frac{r}{2}}(x_0)$, $M_{\frac{15r}{16}}(x_0),$ $M_r(x_0)$, for $x_0\in\R^n$ and $r\in(0,1]$ follows by a translation and anisotropic scaling argument introduced in \Cref{sec:hoeld}. See also \cite{CK20}.

Let us comment on related results in the literature. 
The underlying ideas in developing regularity results for uniformly elliptic operators in divergence form with bounded and measurable coefficients go back to the influential contributions by De Giorgi, Nash and Moser (See \cite{DEGIORGI, NASH, MOSER}). These works led to many further results in various directions.
Similar results for nonlocal operators in divergence form have been obtained by several authors including the works \cite{Kin20, Bar09, BLH, CAFFVASS, KUMA, CHENKUMAWANG,Coz17, DCKP16, DyKa17, FallReg, KASFELS,  kassmann-apriori, KassSchwab, KuMiSi15, Rod21, Ming11, Mosconi,  Str17, Str18}. See also the references therein. For further regularity results concerning nonlocal equations governed by fractional $p$-Laplacians, we refer the reader to \cite{KuMiSi15, ToIrFe15, NOW20, NOW21, NOW21b, NOW21c, BrLiSc18}, \\
In \cite{DCKP16}, the authors extend the De Giorgi--Nash--Moser theory to a class of fractional $p$-Laplace equations. They provide the existence of a unique minimizer to homogenous equations and prove local regularity estimates for weak solutions. Moreover, in \cite{DCKP14}, the same authors prove a general Harnack inequality for weak solutions. 

Nonlocal operators with anisotropic and singular kernels of the type $\ma$ are studied in various mathematical areas such as stochastic differential equations and potential theory. In \cite{BaCh10}, the authors study regularity estimates for harmonic functions for systems of stochastic differential equations $\d X_t=A(X_{t-})\d Z_t$ driven by L\'evy processes $Z_t$ with L\'evy measure $\ma(0,\d y)$, where $2s_1=\dots=2s_n=\alpha$ and $p=2$. See also \cite{ZHANG15, Cha20, ROV16}. Sharp two sided bounds for the heat kernels are established in \cite{KuRy17, KKK19}. In \cite{KuRy20}, the authors prove the existence of transition density of the process $X_t$ and establish semigroup properties of solutions. 
The existence of densities for solutions to stochastic differential equations with H\"older continuous coefficients driven by L\'evy processes with anisotropic jumps has been proved in \cite{FRIPERU18}.
Such type of anisotropies also appear in the study of the anisotropic stable JCIR process, see \cite{FriPen20}.

Our approach follows mainly the ideas of \cite{DK20,CK20} and \cite{CKW19}.
In \cite{DK20}, the authors develop a local regularity theory for a class of linear nonlocal operators which covers the case $s=s_1=\dots=s_n\in(0,1)$ and $p=2$. Based on the ideas of \cite{DK20}, the authors in \cite{CK20} establish regularity estimates in the case $p=2$ for weak solutions in a more general framework which allows the orders of differentiability $s_1,\dots,s_n$ to be different. In \cite{CKW19} parabolic equations in the case $p=2$ and  possible different orders of differentiability are studied. That paper provides robust regularity estimates, which means the constants in the weak Harnack inequality and H\"older regularity estimate do not depend on the orders of differentiability but on their lower one, only. This allows us to recover regularity results for local operators from the theory of nonlocal operators by considering the limit.

The purpose of this paper is to provide local regularity estimates as in \cite{CK20} for operators which are allowed to be nonlinear. This nonlinearity leads to several difficulties like the need for a different proof for the discrete gradient estimate (See \Cref{lem:alg_ineq}). Since we cannot use the helpful properties of Hilbert spaces (like Plancherel's theorem), we also need an approach different from the one in \cite{CK20} to prove a Sobolev-type inequality.
One strength of this paper is the robustness of all results. This allows us to recover regularity estimates for the limit operators such as for the orthotropic $p$-Laplacian. 

Finally, we would like to point out that it is also interesting to study such
operators in non-divergence form. We refer the reader to \cite{SilvInd} for regularity results concerning the fractional Laplacian and to \cite{Lind16} for the fractional $p$-Laplacian. See also \cite{LeitSan20} for the anisotropic case. \\
Even in the most simple case, that is $p=2$ and $s=s_1=s_2=\dots=s_n$, regularity estimates for operators in non-divergence form of the type \eqref{def:nonlocaloperator} with $\mu=\ma$ lead to various open problems such as an Alexandrov--Bakelmann--Pucci estimate.

The authors wish to express their thanks to Lorenzo Brasco for helpful comments.

\subsection*{Outline} This paper is organized as follows. In \Cref{sec:aux}, we introduce appropriate cut-off functions and prove auxiliary results concerning functionals for admissible families of measures. One main result of that section is a Sobolev-type inequality (See \Cref{thm:sobolev}). In \Cref{sec:weak}, we prove the weak Harnack inequality and \Cref{sec:hoeld} contains the proof of the local H\"older estimate. In \Cref{sec:ineq}, we prove some auxiliary algebraic inequalities, and \autoref{sec:anisorect}, we briefly sketch the construction of appropriate anisotropic ``dyadic" rectangles. In \Cref{sec:sharp_maximal}, we use the anisotropic dyadic rectangles to sketch the proof of a suitable sharp maximal function theorem.

\section{Auxiliary results}\label{sec:aux}
This section is devoted to providing some general properties for the class of nonlocal operators that we study in the scope of this paper. The main auxiliary result is a robust Sobolev-type inequality.
  
Let us first introduce a class of suitable cut-off functions that will be useful for appropriate localization.

\begin{definition} \label{def:cut-off}
We say that $(\tau_{x_0,r,\lambda})_{x_0,r,\lambda} \subset C^{0,1}(\R^n)$ is an admissible family of cut-off functions if there is 
$c \ge 1$ such that for all $x_0 \in \R^n$, $r \in (0,1]$ and $\lambda \in (1,2]$, it holds that
\[ \begin{cases}
\supp(\tau) \subset M_{\lambda r}(x_0), \\
\| \tau \|_{\infty} \le 1, \\
\tau \equiv 1 \text{ on } M_r(x_0), \\
\| \partial_k \tau \|_{\infty} \le c \left( \lambda^{s_{\max}/s_k} -1 \right)^{-1}r^{-s_{\max}/s_k} \text{ for every } k \in \lbrace 1 \dots n \rbrace.
\end{cases} \]
\end{definition}

For brevity, we simply write $\tau$ for any such function from $(\tau_{x_0,r,\lambda})_{x_0,r,\lambda}$, if the respective choice of $x_0, r$ and $\lambda$ is clear. The existence of such functions is standard. \\
Recall the definition of admissible families of measures $\cK(p,s_0,\Lambda)$ from \Cref{def:admissible}
\begin{lemma}\label{lemma:cutoff}
Let $p> 1$, $\Lambda \ge 1$, and $s_1,\dots,s_n\in[s_0,1)$ be given for some $s_0\in(0,1)$. There is $C = C(n,p,\Lambda) > 0$ such that for each $\mu\in\cK(p,s_0,\Lambda)$, every $x_0\in\R^n$, $r \in (0,1]$, $\lambda \in (1,2]$ and every admissible cut-off function $\tau$, the following is true:
\begin{equation*}
\sup_{x\in\R^n} \int_{\mathbb{R}^n} |\tau(y)-\tau(x)|^p \mu(x,\mathrm{d}y) \leq C \left( \sum_{k=1}^n \left( \lambda^{s_{\max}/s_k}-1 \right)^{-ps_k} \right) r^{-ps_{\max}}.
\end{equation*}
\end{lemma}
\begin{proof}
We skip the proof. One can follow the lines of the proof from \cite[Lemma 3.1]{CKW19} and will get the same result with the factor $n^{p-1}$ instead of $n$.
\end{proof}
For future purposes, we deduce the following observation. It is an immediate consequence of the foregoing lemma.
\begin{corollary}\label{cor:quadrat}
Let $p> 1$, $\Lambda \ge 1$, and $s_1,\dots,s_n\in[s_0,1)$ be given for some $s_0\in(0,1)$. There is a constant $C = C(n,p,\Lambda) > 0$ such that for each $\mu \in \cK(p,s_0,\Lambda)$ and every $x_0 \in \R^n$, $r \in (0,1]$, $\lambda \in (1,2]$ and every admissible cut-off function $\tau$ and every $u \in L^p(M_{\lambda r}(x_0))$, it holds true that
\begin{align*}
  \int_{M_{\lambda r}(x_0)} \int_{\R^n \setminus M_{\lambda r}(x_0)} &|u(x)|^p|\tau(x)|^p \mu(x,\d y) \d x \\
  &\le C \left( \sum_{k=1}^n (\lambda^{s_{\max}/s_k} - 1)^{-ps_k}\right) r^{-ps_{\max}} \Vert u \Vert_{L^p(M_{\lambda r}(x_0))}^p.
\end{align*}
\end{corollary}
Note that the constants in \Cref{lemma:cutoff} and \Cref{cor:quadrat} do not depend on $s_0$. Therefore, the lower bound $s_0 \leq s_k$ for all $k \in \lbrace 1,\cdots, n\rbrace$ can be dropped here.

\subsection{Functional inequalities}
This subsection is devoted to the proofs of a Sobolev and a Poincar\'{e}-type inequality.
We start our analysis by first proving a technical lemma, see also \cite[Lemma 4.1]{CKW19}.

\begin{lemma} \label{lem:cv}
Let $p >1$, $a \in (0,1]$, $b\geq 1$, $N \in \mathbb{N}$, $k \in \lbrace 1,\cdots, n \rbrace$, and $s_k \in (0,1)$. 
For any $u \in L^p(\mathbb{R}^n)$
\begin{equation*}
\begin{split}
&\int_{\mathbb{R}^n} \sup_{\rho > 0} \frac{1}{\rho^{(1+ps_k)b}} \int_{\mathbb{R}} |u(x) - u(x+he_k)|^p {\bf 1}_{[a\rho^{b},2a\rho^{b})} (|h|) \,\mathrm{d}h \,\mathrm{d}x \\
&\leq (2a)^{1+ps_k} N^{p(1-s_k)} \int_{\mathbb{R}^n} \int_{\mathbb{R}} \frac{|u(x) - u(x+he_k)|^p}{|h|^{1+ps_k}} {\bf 1}_{[\frac{a}{N}\rho^{b}, \frac{2a}{N}\rho^{b})} (|h|) \,\mathrm{d}h\,\mathrm{d}x.
\end{split}
\end{equation*}
\end{lemma}

\begin{proof}
Let $I_a = [a\rho^{b}, 2a\rho^{b})$. By the triangle inequality and a simple change of variables, we have
\begin{equation*}
\begin{split}
&\int_{\mathbb{R}} |u(x) - u(x+he_k)|^p {\bf 1}_{I_a} (|h|) \,\mathrm{d}h \\
&\leq N^{p-1} \sum_{j=1}^N \int_{\mathbb{R}} \left| u\left( x + \frac{j-1}{N} he_k \right) - u\left(x+\frac{j}{N}he_k\right) \right|^p {\bf 1}_{I_a} (|h|) \,\mathrm{d}h \\
&= N^p \sum_{j=1}^N \int_{\mathbb{R}} |u(x + (j-1)he_k) - u(x+jhe_k)|^p {\bf 1}_{I_{a/N}} (|h|) \,\mathrm{d}h.
\end{split}
\end{equation*}
Since $|h| < \frac{2a}{N}\rho^{b}$, we obtain
\begin{equation*}
\begin{split}
&\int_{\mathbb{R}^n} \sup_{\rho > 0} \frac{1}{\rho^{(1+ps_k)b}} \int_{\mathbb{R}} |u(x) - u(x+he_k)|^p {\bf 1}_{I_a} (|h|) \,\mathrm{d}h \,\mathrm{d}x \\
&\leq N^p \left( \frac{2a}{N} \right)^{1+ps_k} \sum_{j=1}^N \int_{\mathbb{R}^n} \int_{\mathbb{R}} \frac{|u(x + (j-1)he_k) - u(x+jhe_k)|^p}{|h|^{1+ps_k}} {\bf 1}_{I_{a/N}} (|h|) \,\mathrm{d}h\,\mathrm{d}x.
\end{split}
\end{equation*}
We change the order of integration by Fubini's theorem and then use the change of variables $y=x+(j-1)he_k$ to conclude that
\begin{equation*}
\begin{split}
&\sum_{j=1}^N \int_{\mathbb{R}^n} \int_{\mathbb{R}} \frac{|u(x + (j-1)he_k) - u(x+jhe_k)|^p}{|h|^{1+ps_k}} {\bf 1}_{I_{a/N}} (|h|) \,\mathrm{d}h\,\mathrm{d}x \\
&= \sum_{j=1}^N \int_{\mathbb{R}} \int_{\mathbb{R}^n} \frac{|u(x + (j-1)he_k) - u(x+jhe_k)|^p}{|h|^{1+ps_k}} {\bf 1}_{I_{a/N}} (|h|) \,\mathrm{d}x\,\mathrm{d}h \\
&= \sum_{j=1}^N \int_{\mathbb{R}} \int_{\mathbb{R}^n} \frac{|u(y) - u(y+he_k)|^p}{|h|^{1+ps_k}} {\bf 1}_{I_{a/N}} (|h|) \,\mathrm{d}y\,\mathrm{d}h \\
&= N \int_{\mathbb{R}^n} \int_{\mathbb{R}} \frac{|u(y) - u(y+he_k)|^p}{|h|^{1+ps_k}} {\bf 1}_{I_{a/N}} (|h|) \,\mathrm{d}h\,\mathrm{d}y.
\end{split}
\end{equation*}
\end{proof}
Using the foregoing result for $b=s_{\max}/s_k$ allows us to prove a robust Sobolev-type inequality. Robust in this context means that the appearing constant in the Sobolev-type inequality is independent of $s_1,\dots,s_n$ and depends on the lower bound $s_0$ only.

Before we prove a robust Sobolev-type inequality, we recall the definition of the Hardy--Littlewood maximal function and sharp maximal function. For $u \in L^1_{\mathrm{loc}}(\mathbb{R}^n)$,
\begin{equation*}
{\bf M}u(x) = \sup_{\rho > 0} \fint_{M_\rho(x)} u(y) \,\d y \quad\text{and}\quad {\bf M}^{\sharp}u(x) = \sup_{\rho > 0} \fint_{M_\rho(x)} |u(y) - (u)_{M_{\rho}(x)}| \,\d y,
\end{equation*}
where $(u)_\Omega = \fint_\Omega u(z)\,\d z$. We will use the maximal function theorem and the sharp maximal function theorem. Note that $\mathbb{R}^n$ is equipped with the metric induced by rectangles of the form $M_r(x)$ and the standard Lebesgue measure. Since $|M_{2r}| = 2^n (2r)^{ns_{\max}/\bar{s}} \leq 2^{n/s_0} |M_r|$, this space is a doubling space with the doubling constant $2^{n/s_0}$.
\begin{theorem} \cite[Theorem 2.2]{Hein01} \label{thm:maximal}
Let $s_1, \dots, s_n \in [s_0, 1)$ be given for some $s_0 \in (0,1)$. Then, there is a constant $C_1 = C_1(n, s_0) > 0$ such that
\begin{equation*}
|\lbrace x \in \mathbb{R}^n : {\bf M}u(x) > t \rbrace| \leq \frac{C_1}{t} \|u\|_{L^1(\mathbb{R}^n)}
\end{equation*}
for all $t > 0$ and $u \in L^1(\mathbb{R}^n)$. For $p > 0$, there is a constant $C_p = C_p(n, p, s_0) > 0$ such that
\begin{equation*}
\|{\bf M}u(x)\|_{L^p(\mathbb{R}^n)} \leq C_p \|u\|_{L^p(\mathbb{R}^n)}
\end{equation*}
for all $u \in L^p(\mathbb{R}^n)$.
\end{theorem}

We were not able to find a reference for the sharp maximal function theorem for sets of the type $M_{\rho}$. Actually, we are not sure whether such result is available in the literature. However, one can follow the ideas of \cite[Section 3.4]{GrafakosMF}, where the $L^p$ bound is established for the sharp maximal function with cubes (instead of anisotropic rectangles). In order to prove the same result for the sharp maximal function with anisotropic rectangles, dyadic cubes have to be replaced by appropriate anisotropic ``dyadic" rectangles. We construct the anisotropic dyadic rectangles in \Cref{sec:anisorect} and prove the following theorem in \Cref{sec:sharp_maximal}. See also \Cref{sec:sharp_maximal} for the definition of the dyadic maximal function ${\bf M}_d u$.
\begin{theorem} \label{thm:sharp_maximal}
Let $s_1, \dots, s_n \in [s_0, 1)$ be given for some $s_0 \in (0,1)$ and let $0 < p_0 \leq p < \infty$. Then, there is a constant $C = C(n, p, s_0) > 0$ such that for all $u \in L^1_{\mathrm{loc}}(\mathbb{R}^n)$ with ${\bf M}_du \in L^{p_0}(\mathbb{R}^n)$,
\begin{equation*}
\|u\|_{L^p(\mathbb{R}^n)} \leq C \|{\bf M}^{\sharp}u\|_{L^p(\mathbb{R}^n)}.
\end{equation*}
\end{theorem}
We are now in a position to prove a robust Sobolev-type inequality by using \Cref{lem:cv}, \Cref{thm:maximal}, and \Cref{thm:sharp_maximal}.

\begin{theorem}\label{thm:sobolev}
Let $s_1,\dots,s_n\in[s_0,1)$ be given for some $s_0\in(0,1)$. Suppose that $1 < p < n/\bar{s}$ and let $p_{\star} = np/(n-p\bar{s})$. Then, there is a constant $C = C(n, p, p_{\star}, s_0) > 0$ such that
for every $u \in V^{p,\mu_\mathrm{axes}}(\mathbb{R}^n | \mathbb{R}^n)$
\begin{equation} \label{eq:sobolev}
\|u\|_{L^{p_{\star}}(\mathbb{R}^n)}^p \leq C \int_{\mathbb{R}^n} \int_{\mathbb{R}^n} |u(x)-u(y)|^p \mu_{\mathrm{axes}}(x, \mathrm{d}y) \mathrm{d}x.
\end{equation}
\end{theorem}

\begin{proof}
This proof is based on the technique of \cite{NDN20}, which uses the maximal and sharp maximal inequalities. Note that by definition of $V^{p,\mu_\mathrm{axes}}(\mathbb{R}^n | \mathbb{R}^n)$ and Hölder's inequality,  $V^{p,\mu_\mathrm{axes}}(\mathbb{R}^n | \mathbb{R}^n)\subset L^p(\R^n) \subset  L^p_{\mathrm{loc}}(\mathbb{R}^n) \subset  L^1_{\mathrm{loc}}(\mathbb{R}^n)$. 
Hence, the maximal and sharp maximal functions are well defined for every function $u \in V^{p,\mu_\mathrm{axes}}(\mathbb{R}^n | \mathbb{R}^n)$.
For $x \in \mathbb{R}^n$ and $\rho > 0$, we have
\begin{equation} \label{eq:avg}
\fint_{M_\rho(x)} |u(y) - (u)_{M_\rho(x)}| \,\mathrm{d}y \leq \fint_{M_\rho(x)} \fint_{M_\rho(x)} |u(y) - u(z)| \,\mathrm{d}z \,\mathrm{d}y.
\end{equation}
Let us consider as in \cite[Lemma 2.1]{CK20} a polygonal chain $\ell = (\ell_0(y,z), \cdots, \ell_n(y,z)) \in \mathbb{R}^{n(n+1)}$ connecting $y$ and $z$ with
\begin{equation*}
\ell_k(y,z) = (l_1^k, \cdots, l_n^k), \quad\text{where}~
l_j^k =
\begin{cases}
z_j, &\text{if}~ j \leq k, \\
y_j, &\text{if}~ j > k,
\end{cases}
\end{equation*}
then $y=\ell_0(y,z)$, $z = \ell_n(y,z)$, and $|\ell_{k-1}(y,z) - \ell_k(y,z)| = |y_k-z_k|$ for all $k=1,\cdots, n$. By the triangle inequality, we have
\begin{equation} \label{eq:tri_ineq}
\begin{split}
&\fint_{M_\rho(x)} \fint_{M_\rho(x)} |u(y) - u(z)| \,\mathrm{d}z \,\mathrm{d}y \\
&\leq \sum_{k=1}^n \fint_{M_\rho(x)} \fint_{M_\rho(x)} |u(\ell_{k-1}(y,z)) - u(\ell_k(y,z))| \,\mathrm{d}z\,\mathrm{d}y.
\end{split}
\end{equation}
For a fixed $k$, we set $w=\ell_{k-1}(y,z) = (z_1,\cdots, z_{k-1}, y_k, \cdots, y_n)$ and $v = y+z-w = (y_1,\cdots, y_{k-1}, z_k, \cdots, z_n)$, then $\ell_k(y,z) = w+e_k(v_k-w_k)$. By Fubini's theorem, we obtain
\begin{equation} \label{eq:Fubini}
\begin{split}
&\fint_{M_\rho(x)} \fint_{M_\rho(x)} |u(\ell_{k-1}(y,z)) - u(\ell_k(y,z))| \,\mathrm{d}z\,\mathrm{d}y \\
&\leq \fint_{M_{\rho}(x)} \fint_{x_k-\rho^{s_{\max}/s_k}}^{x_k+\rho^{s_{\max}/s_k}} |u(w) - u(w+e_k(v_k-w_k))| \,\mathrm{d}v_k\,\mathrm{d}w.
\end{split}
\end{equation}
Moreover, using the inequality $|v_k-w_k| \leq |v_k-x_k| + |w_k-x_k| < 2\rho^{s_{\max}/s_k}$, we make the inner integral in the right-hand side of \eqref{eq:Fubini} independent of $x$. Namely, we have
\begin{equation} \label{eq:x_indep}
\begin{split}
&\fint_{x_k-\rho^{s_{\max}/s_k}}^{x_k+\rho^{s_{\max}/s_k}} |u(w) - u(w+e_k(v_k-w_k))| \,\mathrm{d}v_k \\
&\leq 2 \fint_{w_k-2\rho^{s_{\max}/s_k}}^{w_k+2\rho^{s_{\max}/s_k}} |u(w) - u(w+e_k(v_k-w_k))| \,\mathrm{d}v_k \\
&= 2 \fint_{-2\rho^{s_{\max}/s_k}}^{2\rho^{s_{\max}/s_k}} |u(w) - u(w+he_k)| \,\mathrm{d}h.
\end{split}
\end{equation}
Combining \eqref{eq:avg}, \eqref{eq:tri_ineq}, \eqref{eq:Fubini}, and \eqref{eq:x_indep}, we arrive at
\begin{equation} \label{eq:F_k}
\fint_{M_\rho(x)} |u(y) - (u)_{M_\rho(x)}| \,\mathrm{d}y \leq \sum_{k=1}^n \rho^{s_{\max}} \fint_{M_{\rho}(x)} F_k(w) \,\mathrm{d}w,
\end{equation}
where the function $F_k$ is defined by
\begin{equation*}
F_k(w) := \sup_{\rho > 0} \left( 2\rho^{-s_{\max}} \fint_{-2\rho^{s_{\max}/s_k}}^{2\rho^{s_{\max}/s_k}} |u(w) - u(w+he_k)| \,\mathrm{d}h \right).
\end{equation*}
By H\"older's inequality,
\begin{equation} \label{eq:F_k_Holder}
\begin{split}
\left( \fint_{M_\rho(x)} F_k(w) \,\mathrm{d}w \right)^{p_{\star}} 
&\leq \left( \fint_{M_{\rho}(x)} F_k^p(w) \,\mathrm{d}w \right)^{\frac{{p_{\star}}-p}{p}} \left( \fint_{M_{\rho}(x)} F_k(w) \,\mathrm{d}w \right)^p \\
&\leq |M_{\rho}|^{-\frac{{p_{\star}}-p}{p}} \|F_k\|_{L^p(\mathbb{R}^n)}^{{p_{\star}}-p} \left( \fint_{M_{\rho}(x)} F_k(w) \,\mathrm{d}w \right)^p.
\end{split}
\end{equation}
Thus, it follows from \eqref{eq:F_k} and \eqref{eq:F_k_Holder} that
\begin{equation*}
\begin{split}
\Bigg( \fint_{M_\rho(x)} |u(y) - (u)_{M_\rho(x)}|&\,\mathrm{d}y \Bigg)^{p_{\star}} 
\leq n^{{p_{\star}}-1} \sum_{k=1}^n \rho^{{p_{\star}}s_{\max}} \left( \fint_{M_{\rho}(x)} F_k(w) \,\mathrm{d}w \right)^{p_{\star}} \\
&\leq n^{{p_{\star}}-1} \sum_{k=1}^n \rho^{{p_{\star}}s_{\max}} |M_{\rho}|^{-\frac{{p_{\star}}-p}{p}} \|F_k\|_{L^p(\mathbb{R}^n)}^{{p_{\star}}-p} \left( \fint_{M_{\rho}(x)} F_k(w) \,\mathrm{d}w \right)^p \\
&\leq \frac{n^{{p_{\star}}-1}}{2^{n\frac{{p_{\star}}-p}{p}}} \sum_{k=1}^n \|F_k\|_{L^p(\mathbb{R}^n)}^{{p_{\star}}-p} \left( \fint_{M_{\rho}(x)} F_k(w) \,\mathrm{d}w \right)^p.
\end{split}
\end{equation*}
Taking the supremum over $\rho > 0$, we obtain
\begin{equation} \label{eq:sharp-maximal-upper}
\left( {\bf M}^\sharp u(x) \right)^{p_{\star}} \leq \frac{n^{{p_{\star}}-1}}{2^{n\frac{{p_{\star}}-p}{p}}} \sum_{k=1}^n \|F_k\|_{L^p(\mathbb{R}^n)}^{{p_{\star}}-p} ({\bf M}F_k(x))^p.
\end{equation}
We now use \Cref{thm:maximal} and \Cref{thm:sharp_maximal}. By \eqref{eq:MdM} and \Cref{thm:maximal}, we know that ${\bf M}_d u \in L^p(\mathbb{R}^n)$. Thus, \Cref{thm:sharp_maximal} yields that
\begin{equation*}
\|u\|_{L^{p_{\star}}(\mathbb{R}^n)}^{p_{\star}} \leq C \|{\bf M}^\sharp u\|_{L^{p_{\star}}(\mathbb{R}^n)}^{p_{\star}}
\end{equation*}
for some $C = C(n, p_{\star}, s_0) > 0$. Moreover, assuming $F_k \in L^p(\mathbb{R}^n)$, we have by \Cref{thm:maximal} and \eqref{eq:sharp-maximal-upper}
\begin{equation*}
\|{\bf M}^\sharp u\|_{L^{p_{\star}}(\mathbb{R}^n)}^{p_{\star}} \leq C \sum_{k=1}^n \|F_k\|_{L^p(\mathbb{R}^n)}^{{p_{\star}}-p} \|{\bf M}F_k\|_p^p \leq C \sum_{k=1}^n \|F_k\|_{L^p(\mathbb{R}^n)}^{p_{\star}}
\end{equation*}
for some $C = C(n, p, p_{\star}, s_0) > 0$. Therefore, it only remains to show that
\begin{equation} \label{eq:claim}
\|F_k\|_{L^p(\mathbb{R}^n)}^p \leq C s_k(1-s_k) \int_{\mathbb{R}^n} \int_{\mathbb{R}} \frac{|u(x)-u(x+he_k)|^p}{|h|^{1+ps_k}} \,\mathrm{d}h \,\mathrm{d}x
\end{equation}
for each $k=1,\cdots, n$.

Let us fix $k$. Using H\"older's inequality, we have
\begin{equation*}
\begin{split}
\|F_k\|_{L^p(\mathbb{R}^n)}^p
&\leq \int_{\mathbb{R}^n} \sup_{\rho > 0} \frac{2^p}{\rho^{ps_{\max}}} \fint_{-2\rho^{s_{\max}/s_k}}^{2\rho^{s_{\max}/s_k}} |u(x) - u(x+he_k)|^p \,\mathrm{d}h \,\mathrm{d}x \\
&= \sum_{i=0}^\infty \int_{\mathbb{R}^n} \sup_{\rho > 0} \frac{2^{p-2}}{\rho^{(1+ps_k)s_{\max}/s_k}} \int_{\mathbb{R}} |u(x) - u(x+he_k)|^p {\bf 1}_{I_i}(|h|) \,\mathrm{d}h \,\mathrm{d}x,
\end{split}
\end{equation*}
where $I_i = [2^{-i}\rho^{s_{\max}/s_k}, 2^{-i+1}\rho^{s_{\max}/s_k})$. For each $i$, let $\lbrace \beta_{j,i} \rbrace_{j=0}^\infty$ be a sequence such that $\sum_j \beta_{j,i} \geq 1$, which will be chosen later. Then,
\begin{equation*}
\|F_k\|_{L^p(\mathbb{R}^n)}^p \leq \sum_{i,j=0}^\infty \beta_{j,i} \int_{\mathbb{R}^n} \sup_{\rho > 0} \frac{2^{p-2}}{\rho^{(1+ps_k)s_{\max}/s_k}} \int_{\mathbb{R}} |u(x) - u(x+he_k)|^p {\bf 1}_{I_i}(|h|) \,\mathrm{d}h \,\mathrm{d}x.
\end{equation*}
By \Cref{lem:cv} for $N = 2^j$, $a = 2^{-i} \in (0,1]$ and $b=s_{\max}/s_k$, we obtain
\begin{equation*}
\|F_k\|_{L^p(\mathbb{R}^n)}^p \leq \sum_{i,j=0}^\infty 2^{p-2+(1+ps_k)(1-i)+p(1-s_k)j} \beta_{j,i} \int_{\mathbb{R}^n} \int_{\mathbb{R}} \frac{|u(x) - u(x+he_k)|^p}{|h|^{1+ps_k}} {\bf 1}_{I_{i+j}} (|h|) \,\mathrm{d}h \,\mathrm{d}x.
\end{equation*}
We rearrange the double sums to have
\begin{equation*}
\begin{split}
&\|F_k\|_{L^p(\mathbb{R}^n)}^p \\
&\leq \sum_{i=0}^\infty \sum_{j=0}^i 2^{p-2+(1+ps_k)(1-i+j)+p(1-s_k)j} \beta_{j,i-j} \int_{\mathbb{R}^n} \int_{\mathbb{R}} \frac{|u(x) - u(x+he_k)|^p}{|h|^{1+ps_k}} {\bf 1}_{I_i}(|h|) \,\mathrm{d}h \,\mathrm{d}x.
\end{split}
\end{equation*}
Let $\beta_{j,i} = p(\log 2)(1-s_k) 2^{-p(1-s_k)j}$, then
\begin{equation*}
1 \leq \sum_{j=0}^\infty \beta_{j,i} = \frac{p(\log 2)(1-s_k)}{1-2^{-p(1-s_k)}} \leq 2p < +\infty.
\end{equation*}
Since
\begin{equation*}
\begin{split}
\sum_{j=0}^i 2^{p-2+(1+ps_k)(1-i+j)+p(1-s_k)j} \beta_{j,i-j} 
&= p(\log 2)(1-s_k) 2^{p-2+(1+ps_k)(1-i)} \sum_{j=0}^i 2^{(1+ps_k)j} \\
&= p(\log 2)(1-s_k) 2^{p-2+(1+ps_k)(1-i)} \frac{2^{(1+ps_k)(i+1)}}{2^{1+ps_k}-1} \\
&\leq p(1-s_k) 2^{p(1+s_k)},
\end{split}
\end{equation*}
we arrive at
\begin{equation*}
\begin{split}
\|F_k\|_{L^p(\mathbb{R}^n)}^p
&\leq \sum_{i=0}^\infty p(1-s_k)2^{p(1+s_k)} \int_{\mathbb{R}^n} \int_{\mathbb{R}} \frac{|u(x) - u(x+he_k)|^p}{|h|^{1+ps_k}} {\bf 1}_{I_i}(|h|) \,\mathrm{d}h \,\mathrm{d}x \\
&\leq p2^{2p}\frac{s_k}{s_0}(1-s_k) \int_{\mathbb{R}^n} \int_{\mathbb{R}} \frac{|u(x) - u(x+he_k)|^p}{|h|^{1+ps_k}} \,\mathrm{d}h \,\mathrm{d}x,
\end{split}
\end{equation*}
which proves \eqref{eq:claim}.
\end{proof}

Next, we can make use of appropriate cut-off functions to prove a localized version of the foregoing Sobolev-type inequality.

\begin{corollary} \label{thm:sobolev_loc}
Let $\Lambda \ge 1$ and $s_1,\dots,s_n\in[s_0,1)$ be given for some $s_0\in(0,1)$. Suppose that $1 < p < n/\bar{s}$. There is $C = C(n,p,p_{\star},s_0,\Lambda) > 0$ such that for each $\mu \in \cK(p,s_0,\Lambda)$ and every $x_0 \in \R^n$, $r \in (0,1]$, $\lambda \in (1,2]$ and $u \in H^{p,\mu}_{{M_{\lambda r}(x_0)}}(\R^n)$ it holds
\begin{equation*}
\begin{split}
\|u\|_{L^{{p_{\star}}}(M_r(x_0))}^p\
\leq& C \int_{M_{\lambda r}(x_0)} \int_{M_{\lambda r}(x_0)} |u(x)-u(y)|^p \mu(x, \mathrm{d}y) \mathrm{d}x \\
&+ C \left( \sum_{k=1}^n \left( \lambda^{s_{\max}/s_k}-1 \right)^{-s_k p} \right) r^{-ps_{\max}} \|u\|_{L^p(M_{\lambda r}(x_0))}^p,
\end{split}
\end{equation*}
where ${p_{\star}}$ is defined as in \Cref{thm:sobolev}.
\end{corollary}
\begin{proof}
 Let $\tau:\R^n\to\R$ be an admissible cut-off function in the sense of \Cref{def:cut-off}.
For simplicity of notation we write $M_r=M_r(x_0)$.
 
By \Cref{thm:sobolev} there is a constant $c_1=c_1(n,p,p_{\star},s_0)>0$ such that
 \begin{align*}
  \|u\tau\|_{L^{{p_{\star}}}(\R^n)}^p & \leq  c_1 \Bigg(\int_{M_{\lambda 
r}}\int_{M_{\lambda r}} |u(x)\tau(x)-u(x)\tau(y)|^p \, \ma(x,\d y)\, \d x \\
  & \qquad \qquad + 2\int_{M_{\lambda r}}\int_{(M_{\lambda r})^c} 
|u(x)\tau(x)-u(x)\tau(y)|^p \, \ma(x,\d y)\, \d x \Bigg) \\
  & =: c_1 (I_1 + 2I_2).
 \end{align*}
We have
\begin{align*}
 I_1 & \leq \frac{1}{2^{p}} \Bigg(\int_{M_{\lambda r}}\int_{M_{\lambda r}} 
2^{p-1}\left|(u(y)-u(x))(\tau(x)+\tau(y))\right|^p \, \ma(x,\d y)\, \d x\\
 & \qquad \quad + \int_{M_{\lambda r}}\int_{M_{\lambda 
r}}2^{p-1}\left|(u(x)+u(y))(\tau(x)-\tau(y))\right|^p \, \ma(x,\d y)\, \d x\Bigg) \\
 & = \frac12 (J_1+J_2),
\end{align*}
Since $(\tau(x)+\tau(y))\leq 2$ for all $x,y\in M_{\lambda r}$, we get
\[
 J_1 \leq  2^p\int_{M_{\lambda r}}\int_{M_{\lambda r}} |u(y)-u(x)|^p\, \ma(x,\d 
y)\,\d x \leq \Lambda 2^p\int_{M_{\lambda r}}\int_{M_{\lambda r}} |u(y)-u(x)|^p\, \mu(x,\d 
y)\,\d x, \]
where we used \Cref{assumption:comparability} in the second inequality.\\
Moreover, since $|u(x)+u(y)|^p|\tau(x)-\tau(x)|^p \leq 2^{p-1}|u(x)|^p|\tau(x)-\tau(x)|^p + 
2^{p-1}|u(y)|^p|\tau(y)-\tau(x)|^p$, we can again apply \Cref{assumption:comparability} and by \Cref{lemma:cutoff}, we get
\begin{align*} 
J_2 &\leq 2^p \left( \sup_{x\in \R^n} \int_{\R^n} |\tau(y)-\tau(x)|^p\mu(x,\d y) \right)\|u\|^p_{L^p(M_{\lambda r})}
 \\
 &\leq c_2 \left(\sum_{k=1}^n 
(\lambda^{\frac{s_{\max}}{s_k}}-1)^{-ps_k}\right)r^{-ps_{\max}}\|u\|^p_{L^p(M_{\lambda 
r})}\,
\end{align*}
for some $c_2>0$, depending on $n$, $p$, $s_0$ and $\Lambda$.
Moreover, by \Cref{cor:quadrat} there is $c_3=c_3(n,p,\Lambda)>0$ such that
\[ I_2  \leq c_3 \left(\sum_{k=1}^n 
(\lambda^{\frac{s_{\max}}{s_k}}-1)^{-ps_k}\right) r^{-ps_{\max}} \|u\|_{L^p(M_{\lambda r})}^p  \]
Combining these estimates, we find a constant $C=C(n,p,p_{\star},s_0,\Lambda)>0$ such that
\begin{align*}
  \|u\|_{L^{{p_{\star}}}(M_r)}^p &\leq \|u\tau\|_{L^{{p_{\star}}}(\R^n)}^p \\
 & \leq C \Bigg(\int_{M_{\lambda r}}\int_{M_{\lambda r}} |u(x)-u(y)|^p \, 
\ma(x,\d y)\, \, \d x\\
&\qquad \qquad \qquad \qquad \qquad   +\left(\sum_{k=1}^n 
(\lambda^{\frac{s_{\max}}{s_k}}-1)^{-ps_k}\right)r^{-ps_{\max}} \|u\|^p_{L^p(M_{\lambda r})}\Bigg). 
\end{align*}
\end{proof}
Applying the same method as in the proof of the Sobolev-type inequality \Cref{thm:sobolev}, we can deduce a Poincar\'e inequality.
\begin{theorem}\label{thm:poincare}
Let $p> 1$, $\Lambda \ge 1$, and $s_1,\dots,s_n\in[s_0,1)$ be given for some $s_0\in(0,1)$. There is $C = C(n,p,s_0,\Lambda) > 0$ such that for each $\mu \in \cK(p,s_0,\Lambda)$ and every $x_0 \in \R^n$, $r \in (0,1]$ and $u \in L^p(M_{r}(x_0))$, 
\begin{equation*}
\|u - (u)_{M_r(x_0)}\|_{L^p(M_r(x_0))}^p \leq C r^{ps_{\max}} \mathcal{E}_{M_r(x_0)}^\mu(u,u).
\end{equation*}
\end{theorem}
The proof is analog to the proof of the Poincar\'e inequality for the case $p=2$, see \cite[Theorem 4.2]{CKW19}.

\section{Weak Harnack inequality}\label{sec:weak}
In this section, we prove \Cref{thm:weak_Harnack}. The proof is based on Moser's iteration technique. We first need to verify a few properties for weak supersolutions to \eqref{eq:PDE}.
\begin{lemma}\label{lemma:log}
Let $\Lambda \ge 1$ and $s_1,\dots,s_n\in[s_0,1)$ be given for some $s_0\in(0,1)$. Let $1<p\leq  n/\bar{s}$, $x_0 \in \R^n$, $r \in (0,1]$, and $\lambda \in (1,2]$. Set $M_r = M_r(x_0)$ and assume $f \in L^{q/(p\bar{s})}(M_{\lambda r})$ for some $q > n$. There is $C = C(n,p,s_0,\Lambda) > 0$ such that for each $\mu \in \mathcal{K}(p, s_0, \Lambda)$ and every $u \in V^{p,\mu}(M_{\lambda r} | \mathbb{R}^n)$ that satisfies
\begin{align*}
&\mathcal{E}^\mu(u, \varphi) \geq (f, \varphi) \quad \text{for any nonnegative}~ \varphi \in H^{p,\mu}_{M_{\lambda r}}(\R^n), \\
&u(x) \geq \epsilon \quad \text{a.e. in}~ M_{\lambda r} ~ \text{for some}~ \epsilon > 0,
\end{align*}
the following holds:
\begin{equation*}
\begin{split}
&\int_{M_r} \int_{M_r} |\log u(y) - \log u(x)|^p \,\mu(x,\d y) \,\d x \\
&\leq C \left( \sum_{k=1}^n \left( \lambda^{s_{\max}/s_k}-1 \right)^{-ps_k} \right) r^{-ps_{\max}} |M_{\lambda r}| \\
&\quad + \epsilon^{1-p} \|f\|_{L^{q/(p\bar{s})}(M_{\lambda r})} |M_{\lambda r}|^{\frac{q-p\bar{s}}{q}} + 2\epsilon^{1-p} |M_{\lambda r}| \sup_{x \in M_{(\lambda+1)r/2}} \int_{\mathbb{R}^n \setminus M_{\lambda r}} u_-^{p-1}(y) \mu(x, \d y).
\end{split}
\end{equation*}
\end{lemma}

\begin{proof}
Let $\tau$ be an admissible cut-off function in the sense of \Cref{def:cut-off} and let $\varphi(x) = \tau^p(x)u^{1-p}(x)$, which is well defined since $\supp(\tau) \subset M_{\lambda r}$. Then, we have
\begin{equation} \label{eq:log_I12}
\begin{split}
(f,\varphi)
&\leq \int_{M_{\lambda r}} \int_{M_{\lambda r}} |u(x)-u(y)|^{p-2}(u(x) - u(y)) \left( \frac{\tau^p(x)}{u^{p-1}(x)} - \frac{\tau^p(y)}{u^{p-1}(y)} \right) \mu(x,\d y) \,\d x \\
&\quad + 2\int_{M_{\lambda r}} \int_{\mathbb{R}^n \setminus M_{\lambda r}} |u(x)-u(y)|^{p-2}(u(x) - u(y)) \frac{\tau^p(x)}{u^{p-1}(x)} \mu(x,\d y) \,\d x \\
&=: I_1 + I_2.
\end{split}
\end{equation}
Similar to the proof of \cite[Lemma 1.3]{DCKP16}, we get the inequality
\begin{equation} \label{eq:ineq_log}
\begin{split}
&|u(x)-u(y)|^{p-2} (u(x)-u(y)) \left( \frac{\tau^p(x)}{u^{p-1}(x)} - \frac{\tau^p(y)}{u^{p-1}(y)} \right) \\
&\leq - c_1 |\log u(x) - \log u(y)|^p \tau^p(y) + c_2 |\tau(x)-\tau(y)|^p,
\end{split}
\end{equation}
where $c_1, c_2 >0$ are constants depending only on $p$. Hence, by \eqref{eq:ineq_log} and \Cref{lemma:cutoff},
\begin{equation} \label{eq:log_I1}
\begin{split}
I_1
&\leq -c_1 \int_{M_r} \int_{M_r} |\log u(x) - \log u(y)|^p \mu(x,\d y) \,\d x \\
&\quad + C \left( \sum_{k=1}^n \left( \lambda^{s_{\max}/s_k}-1 \right)^{-ps_k} \right)r^{-ps_{\max}} |M_{\lambda r}|.
\end{split}
\end{equation}
For $I_2$, again by \Cref{lemma:cutoff}
\begin{equation} \label{eq:log_I2}
\begin{split}
I_2
&\leq 2\int_{M_{\lambda r}} \int_{\mathbb{R}^n \setminus M_{\lambda r}} (u(x)-u(y))^{p-1} \frac{\tau^p(x)}{u^{p-1}(x)} {\bf 1}_{\lbrace u(x) \geq u(y) \rbrace} \mu(x,\d y) \,\d x \\
&\leq 2\int_{M_{\lambda r}} \int_{\mathbb{R}^n \setminus M_{\lambda r}} |\tau(x) - \tau(y)|^p \mu(x,\d y) \,\d x + 2\int_{M_{\lambda r}} \int_{\mathbb{R}^n \setminus M_{\lambda r}} u_-^{p-1}(y) \frac{\tau^p(x)}{\epsilon^{p-1}} \mu(x,\d y) \,\d x \\
&\leq C \left( \sum_{k=1}^n \left( \lambda^{s_{\max}/s_k}-1 \right)^{-ps_k} \right)r^{-ps_{\max}} |M_{\lambda r}| + \frac{2|M_{\lambda r}|}{\epsilon^{p-1}} \sup_{x \in M_{(\lambda+1)r/2}} \int_{\mathbb{R}^n \setminus M_{\lambda r}} u_-^{p-1}(y) \mu(x, \d y),
\end{split}
\end{equation}
where we assumed that $\mathrm{supp}(\tau) \subset M_{(\lambda+1)r/2}$. Combining \eqref{eq:log_I12}, \eqref{eq:log_I1}, and \eqref{eq:log_I2}, and using H\"older's inequality, we conclude that
\begin{equation*}
\begin{split}
&\int_{M_r} \int_{M_r} |\log u(y) - \log u(x)|^p \,\mu(x,\d y) \,\d x \\
&\leq C \left( \sum_{k=1}^n \left( \lambda^{s_{\max}/s_k}-1 \right)^{-ps_k} \right) r^{-ps_{\max}} |M_{\lambda r}| \\
&\quad + \epsilon^{1-p} \|f\|_{L^{q/(p\bar{s})}(M_{\lambda r})} |M_{\lambda r}|^{\frac{q-p\bar{s}}{q}} + 2\epsilon^{1-p} |M_{\lambda r}| \sup_{x \in M_{(\lambda+1)/2}} \int_{\mathbb{R}^n \setminus M_{\lambda r}} u_-^{p-1}(y) \mu(x, \d y).
\end{split}
\end{equation*}
\end{proof}

The next theorem is an essential result to prove the weak Harnack inequality.

\begin{theorem}\label{thm:flipsigns}
Let $\Lambda \ge 1$ and $s_1,\dots,s_n\in[s_0,1)$ be given for some $s_0\in(0,1)$. Let $1<p<  n/\bar{s}$, $x_0 \in \R^n$, and $r \in (0,1]$. Set $M_r = M_r(x_0)$ and assume $f \in L^{q/(p\bar{s})}(M_{5r/4})$ for some $q > n$. There are $C=C(n,p,s_0,q,\Lambda)>0$ and $\bar{p}=\bar{p}(n,p,s_0,q,\Lambda) \in (0, 1)$ such that for each $\mu \in \mathcal{K}(p, s_0, \Lambda)$ and every $u \in V^{p,\mu}(M_{5r/4} | \mathbb{R}^n)$ that satisfies
\begin{align*}
&\mathcal{E}^\mu(u, \varphi) \geq (f, \varphi) \quad \text{for any nonnegative}~ \varphi \in H^{p,\mu}_{M_{5r/4}}(\R^n), \\
&u(x) \geq \epsilon \quad \text{a.e. in}~ M_{5r/4},
\end{align*}
for
\begin{equation*}
\epsilon > r^{\delta}\|f\|_{L^{q/(p\bar{s})}(M_{5r/4})}^{\frac{1}{p-1}} + \left( r^{ps_{\max}} \sup_{x \in M_{9r/8}} \int_{\mathbb{R}^n \setminus M_{5r/4}} u_-^{p-1}(y) \mu(x, \d y) \right)^{\frac{1}{p-1}},
\end{equation*}
where $\delta = \frac{ps_{\max}}{p-1} \frac{q-n}{q}$, the following holds:
\begin{equation*}
\left( \fint_{M_r} u^{\bar{p}}(x) \,\mathrm{d}x \right)^{1/\bar{p}} \leq C \left( \fint_{M_r} u^{-\bar{p}}(x) \,\mathrm{d}x \right)^{-1/\bar{p}}.
\end{equation*}
\end{theorem}

\begin{proof}
We only need to prove that $\log u \in \BMO(M_r)$. The rest of the proof is standard.
The Poincar\'e inequality (see \Cref{thm:poincare}) and \Cref{lemma:log} imply
\begin{equation*}
\begin{split}
&\| \log u - (\log u)_{M_r} \|_{L^p(M_r)}^p \\
&\leq C r^{ps_{\max}} \mathcal{E}^\mu_{M_r} (\log u, \log u) \\
&\leq C \left( \sum_{k=1}^n \left( \left(\frac{5}{4}\right)^{\frac{s_{\max}}{s_k}}-1 \right)^{-s_k p} \right) |M_{5r/4}| + C\epsilon^{1-p} \|f\|_{L^{q/(p\bar{s})}(M_{5r/4})} r^{ps_{\max}} |M_{5r/4}|^{\frac{q-p\bar{s}}{q}} \\
&\quad + C\epsilon^{1-p} r^{ps_{\max}} |M_{5r/4}| \sup_{x \in M_{9r/8}} \int_{\mathbb{R}^n \setminus M_{5r/4}} u_-^{p-1}(y) \mu(x, \d y) \\
&\leq C|M_r|,
\end{split}
\end{equation*}
where we used the bound on $\epsilon$ in the last inequality.
Finally, by H\"older's inequality we obtain
\begin{equation*}
\|\log u\|_{\mathrm{BMO}(M_r)} \leq \left( \fint_{M_r} |\log u - (\log u)_{M_r}|^p \,\d x \right)^{1/p} \leq C,
\end{equation*}
which shows that $\log u \in \mathrm{BMO}(B_r)$. 
\end{proof}

In order to apply Moser's iteration for negative exponents, we prove the following lemma.

\begin{lemma} \label{lem:iteration}
Let $\Lambda \ge 1$ and $s_1,\dots,s_n\in[s_0,1)$ be given for some $s_0\in(0,1)$. Let $1<p<  n/\bar{s},$ $x_0 \in \R^n$, $r\in(0,1]$, and $\lambda \in (1,2]$. Set $M_r = M_r(x_0)$ and assume $f\in L^{q/(p\bar{s})}(M_{\lambda r})$ for some $q>n$. 
For each $\mu \in \mathcal{K}(p,s_0, \Lambda)$ and every $u \in V^{p,\mu}(M_{\lambda r} | \R^n)$ that satisfies 
\begin{align*}
&\mathcal{E}^\mu(u, \varphi) \geq (f, \varphi) \quad \text{for any nonnegative}~ \varphi \in H^{p,\mu}_{M_{\lambda r}}(\R^n), \\
&u(x) \geq \epsilon \quad \text{a.e. in}~ M_{\lambda r},
\end{align*}
for
\begin{equation*}
\epsilon > r^{\delta}\|f\|_{L^{q/(p\bar{s})}(M_{\lambda r/4})}^{\frac{1}{p-1}} + \left( r^{ps_{\max}} \sup_{x \in M_{(\lambda+1)r/2}} \int_{\mathbb{R}^n \setminus M_{\lambda r}} u_-^{p-1}(y) \mu(x, \d y) \right)^{\frac{1}{p-1}},
\end{equation*}
the following is true for any $t > p-1$,
\begin{equation*}
\left\|u^{-1}\right\|_{L^{(t-p+1)\gamma}(M_r)}^{t-p+1} \leq C \left( \sum_{k=1}^n \left( \lambda^{s_{\max}/s_k}-1 \right)^{-s_k p} \right) r^{-s_{\max}p} \left\|u^{-1}\right\|_{L^{t-p+1}(M_{\lambda r})}^{t-p+1},
\end{equation*}
where $\delta = \frac{ps_{\max}}{p-1} \frac{q-n}{q}$, $\gamma = n/(n-p\bar{s})$, and $C = C(n, p, p_{\star}, q, t, s_0, \Lambda) > 0$ is a constant that is bounded when $t$ is bounded away from $p-1$.
\end{lemma}

To prove \Cref{lem:iteration}, we need the following algebraic inequality.

\begin{lemma} \label{lem:alg_ineq}
Let $a, b > 0$, $\tau_1, \tau_2 \in [0,1]$, and $t > p-1 > 0$. Then,
\begin{equation*}
\begin{split}
&|b-a|^{p-2}(b-a)(\tau_1^p a^{-t} - \tau_2^p b^{-t}) \\
&\geq c_1 \left| \tau_1 a^{\frac{-t+p-1}{p}} - \tau_2 b^{\frac{-t+p-1}{p}} \right|^p - c_2 |\tau_1-\tau_2|^p \left( a^{-t+p-1} + b^{-t+p-1} \right),
\end{split}
\end{equation*}
where $c_i = c_i(p, t) > 0$, $i=1, 2$, is bounded when $t$ is bounded away from $p-1$.
\end{lemma}

Note that \Cref{lem:alg_ineq} is a discrete version of 
\begin{equation*}
|\nabla v|^{p-2} \nabla v \cdot \nabla (-v^{-t}\tau^p) \geq c_1 \left| \nabla \left(v^{\frac{-t+p-1}{p}} \tau \right) \right|^p - c_2 |\nabla \tau|^p v^{-t+p-1}.
\end{equation*}
The proof of \Cref{lem:alg_ineq} is provided in \Cref{sec:ineq}.

\begin{proof} [Proof of \Cref{lem:iteration}]
Let $\tau$ be an admissible cut-off function in the sense of \Cref{def:cut-off}. Since $\tau = 0$ outside $M_{\lambda r}$, the function $\varphi = -\tau^p u^{-t}$ is well defined. Using \Cref{lem:alg_ineq}, we have
\begin{equation*}
\begin{split}
&(f, -\tau^p u^{-t}) \geq \mathcal{E}(u, -\tau^p u^{-t}) \\
&= \int_{M_{\lambda r}} \int_{M_{\lambda r}} |u(y)-u(x)|^{p-2} (u(y)-u(x)) \left( \tau^p(x) u^{-t}(x) - \tau^p(y) u^{-t}(y) \right) \mu(x, \d y)\, \d x \\
&\quad + 2 \int_{M_{\lambda r}} \int_{\mathbb{R}^n \setminus M_{\lambda r}} |u(y)-u(x)|^{p-2} (u(y)-u(x)) \tau^p(x) u^{-t}(x) \mu(x, \d y)\, \d x \\
&\geq c_1 \int_{M_{\lambda r}} \int_{M_{\lambda r}} \left| \tau(x) u^{\frac{-t+p-1}{p}} (x) - \tau(y) u^{\frac{-t+p-1}{p}}(y) \right|^p \mu(x, \d y)\, \d x \\
&\quad - c_2 \int_{M_{\lambda r}} \int_{M_{\lambda r}} |\tau(x)-\tau(y)|^p \left( u^{-t+p-1}(x) + u^{-t+p-1}(y) \right) \mu(x, \d y)\, \d x \\
&\quad - 2 \int_{M_{\lambda r}} \int_{\mathbb{R}^n \setminus M_{\lambda r}} (u(x)-u(y))^{p-1} \tau^p(x) u^{-t}(x) {\bf 1}_{\lbrace u(y) \leq u(x) \rbrace} \mu(x, \d y)\, \d x \\
&=: c_1 I_1 - c_2 I_2 - I_3,
\end{split}
\end{equation*}
where $c_1$ and $c_2$ are constants given in \Cref{lem:alg_ineq}. 
By \Cref{thm:sobolev_loc}, we obtain
\begin{equation*}
\begin{split}
I_1
&= \int_{M_{\lambda r}} \int_{M_{\lambda r}} \left| \tau(x) u^{\frac{-t+p-1}{p}} (x) - \tau(y) u^{\frac{-t+p-1}{p}}(y) \right|^p \mu(x, \d y)\, \d x \\
&\geq C\left\|\tau u^{\frac{-t+p-1}{p}} \right\|_{L^{p_{\star}}(M_{r})}^p -C r^{-ps_{\max}}\left( \sum_{k=1}^n \left( \lambda^{s_{\max}/s_k}-1 \right)^{-s_k p} \right) \left\|\tau u^{\frac{-t+p-1}{p}}\right\|_{L^p(M_{\lambda r})}^p,
\end{split}
\end{equation*}
where
$ p_{\star} = \frac{np}{n-p\bar{s}}$.

For $I_2$, we use \Cref{lemma:cutoff} again to have
\begin{equation*}
\begin{split}
I_2
&= 2\int_{M_{\lambda r}} \int_{M_{\lambda r}} |\tau(x)-\tau(y)|^p u^{-t+p-1}(x) \,\mu(x, \d y)\, \d x \\
&\leq C r^{-ps_{\max}}\left( \sum_{k=1}^n \left( \lambda^{s_{\max}/s_k}-1 \right)^{-s_k p} \right) \left\|u^{-t+p-1}\right\|_{L^1(M_{\lambda r})}.
\end{split}
\end{equation*}
For $I_3$, assuming that $\mathrm{supp}(\tau) \subset M_{(\lambda+1)r/2}$ and using \Cref{lemma:cutoff} we deduce
\begin{equation*}
\begin{split}
I_3
&\leq C \int_{M_{\lambda r}} \int_{\mathbb{R}^n \setminus M_{\lambda r}} \left( u^{p-1}(x)+u_-^{p-1}(y) \right) \tau^p(x) u^{-t}(x) \mu(x, \d y)\, \d x \\
&\leq C r^{-ps_{\max}}\left( \sum_{k=1}^n \left( \lambda^{s_{\max}/s_k}-1 \right)^{-s_k p} \right) \left\|u^{-t+p-1}\right\|_{L^1(M_{\lambda r})} \\
&\quad + C\epsilon^{1-p} \left( \sup_{x \in M_{(\lambda+1)r/2}} \int_{\mathbb{R}^n \setminus M_{\lambda r}} u_-^{p-1}(y) \mu(x, \d y) \right) \left\|u^{-t+p-1}\right\|_{L^1(M_{\lambda r})}.
\end{split}
\end{equation*}
Moreover, we estimate
\begin{equation*}
\begin{split}
|(f, -\tau^p u^{-t})|
&\leq \epsilon^{1-p} \int_{\R^n} |f| \tau^p u^{-t+p-1} \,\d x \\
&\leq \epsilon^{1-p} \|f\|_{L^{q/(p\bar{s})}(M_{\lambda r})} \left\| \tau^p u^{-t+p-1} \right\|_{L^{q/(q-p\bar{s})}(M_{\lambda r})} \\
&= \epsilon^{1-p} \|f\|_{L^{q/(p\bar{s})}(M_{\lambda r})} \left\| \tau u^{\frac{-t+p-1}{p}} \right\|_{L^{pq/(q-p\bar{s})}(M_{\lambda r})}^p.
\end{split}
\end{equation*}
Using Lyapunov's inequality and Young's inequality, we have
\begin{equation*}
\|v\|_{pq/(q-p\bar{s})}^p \leq \|v\|_{p_\star}^{np/q} \|v\|_p^{(qp-np)/q} \leq \frac{n}{q} \omega \|v\|_{p_{\star}}^p + \frac{q-n}{q} \omega^{-n/(q-n)} \|v\|_p^p
\end{equation*}
for any $v \in L^{p_{\star}} \cap L^p$ and any $\omega > 0$. This yields that
\begin{equation*}
\begin{split}
|(f, -\tau^p u^{-t})|
&\leq \epsilon^{1-p} \|f\|_{L^{q/(p\bar{s})}(M_{\lambda r})} \left( \frac{n}{q} \omega \left\| \tau u^{\frac{-t+p-1}{p}} \right\|_{L^{p_\star}}^p + \frac{q-n}{q} \omega^{-n/(q-n)} \left\|\tau u^{\frac{-t+p-1}{p}} \right\|_{L^p}^p \right) \\
&\leq r^{-ps_{\max}\frac{q-n}{q}} \left( \frac{n}{q} \omega \left\| \tau^p u^{-t+p-1} \right\|_{L^{\gamma}} + \frac{q-n}{q} \omega^{-n/(q-n)} \left\|\tau^p u^{-t+p-1} \right\|_{L^1} \right).
\end{split}
\end{equation*}
Combining all the estimates, we have
\begin{equation*}
\begin{split}
&\left\|\tau^p u^{-t+p-1} \right\|_{L^{\gamma}(M_{\lambda r})} \\
&\leq C r^{-ps_{\max}}\left( 1+\sum_{k=1}^n \left( \lambda^{s_{\max}/s_k}-1 \right)^{-s_k p} \right) \left\|u^{-t+p-1}\right\|_{L^1(M_{\lambda r})} \\
&\quad +Cr^{-ps_{\max}\frac{q-n}{q}} \left( \frac{n}{q} \omega \left\| \tau^p u^{-t+p-1} \right\|_{L^{\gamma}(M_{\lambda r})} + \frac{q-n}{q} \omega^{-n/(q-n)} \left\|\tau^p u^{-t+p-1} \right\|_{L^1(M_{\lambda r})} \right).
\end{split}
\end{equation*}
Taking $\omega = \varepsilon_0 r^{-ps_{\max}\frac{q-n}{q}}$ with $\varepsilon_0 > 0$ small enough, we arrive at
\begin{equation*}
\begin{split}
\left\| u^{-1} \right\|_{L^{(t-p+1)\gamma}(M_r)}^{t-p+1}
&\leq \left\|\tau^p u^{-t+p-1} \right\|_{L^{\gamma}(M_{\lambda r})} \\
&\leq C\left( 1+\sum_{k=1}^n \left( \lambda^{s_{\max}/s_k}-1 \right)^{-s_k p} \right) r^{-ps_{\max}} \left\|u^{-1}\right\|_{L^{t-p+1}(M_{\lambda r})}^{t-p+1},
\end{split}
\end{equation*}
where $C$ depends on $n$, $p$, $p_{\star}$, $t$, $s_0$, $q$ and $\Lambda$, and is bounded when $t$ is bounded away from $p-1$. Since $\lambda \leq 2$, we obtain
\begin{equation*}
\sum_{k=1}^n \left( \lambda^{s_{\max}/s_k}-1 \right)^{-s_k p} \geq \sum_{k=1}^n \left( \lambda^{1/s_k}-1 \right)^{-s_k p} \geq \sum_{k=1}^n 2^{-p} = n2^{-p},
\end{equation*}
from which the desired result follows.
\end{proof}

The standard iteration technique proves the following lemma, see \cite{CK20,DK20}.

\begin{lemma} \label{lem:inf}
Under the same assumptions as in \Cref{lem:iteration}, for any $p_0 > 0$ there is a constant $C = C(n, p, p_{\star}, q, p_0, s_0, \Lambda) > 0$ such that
\begin{equation}\label{eq:infest}
\inf_{M_r} u \geq C \left( \fint_{M_{2r}} u(x)^{-p_0} \, \mathrm{d}x \right)^{-1/p_0}.
\end{equation}
\end{lemma}

The proof of \Cref{thm:weak_Harnack} follows from \Cref{thm:flipsigns}, \Cref{lem:inf} and the triangle inequality.

\section{H\"older estimates}\label{sec:hoeld}

This section is devoted to the proof of \Cref{thm:Holder}. The general scheme for the derivation of a priori interior H\"older estimates from the weak Harnack inequality in the non-local setting has been developed in \cite{DK20} and applied successfully to the anisotropic setting \cite{CK20} when $p=2$. We extend the result presented in \cite{CK20} to the general case $p > 1$.

Recall that the rectangles in \Cref{def:M_r} satisfy the following property. For $\lambda > 0$ and $\Omega \subset \R^n$ open, we have
\begin{equation*}
\mathcal{E}_\Omega^{\ma}(u\circ \Psi, v \circ \Psi) = \lambda^{-(n-\bar{s}p)s_{\max}/\bar{s}} \mathcal{E}_{\Psi(\Omega)}^{\ma}(u, v) \quad\text{for every}~ u, v \in V^{p,\ma}(\Omega|\R^n)
\end{equation*}
and
\begin{equation*}
(f \circ \Psi, \varphi \circ \Psi) = \lambda^{-ns_{\max}/\bar{s}} (f, \varphi)  \quad\text{for every}~ f \in L^{q/(p\bar{s})}(\Omega),~ \varphi \in H^{p,\ma}_{\Omega}(\R^n),
\end{equation*}
where $\Psi : \R^n \to \R^n$ is a diffeomorphism given by
\begin{equation} \label{eq:Psi}
\Psi(x) =
\begin{pmatrix}
\lambda^{s_{\max}/s_1} & \cdots & 0 \\
\vdots & \ddots & \vdots \\
0 & \cdots & \lambda^{s_{\max}/s_n}
\end{pmatrix}x.
\end{equation}
The rectangles from \Cref{def:M_r} are balls in a metric space $(\R^n,d)$, where the metric $d:\R^n\times\R^n\to[0,\infty)$ is defined as follows:
 \[ d(x,y) = \sup_{k\in\{1,\dots,n\}} |x_k-y_k|^{s_k/s_{\max}}. \]
By the scaling property and covering arguments provided in \cite{CK20}, it is enough to show the following theorem. 

\begin{theorem} \label{thm:Holder_0}
Let $\Lambda \ge 1$ and $s_1,\dots,s_n\in[s_0,1)$ be given for some $s_0\in(0,1)$. Let $1<p<  n/\bar{s}$. Assume $f\in L^{q/(p\bar{s})}(M_1)$ for some $q>n$. There are $\alpha=\alpha(n,p,p_{\star},q,s_0,\Lambda) \in (0,1)$ and $C=C(n,p,p_{\star},q,s_0,\Lambda) > 0$ such that for each $\mu \in \mathcal{K}(p,s_0,\Lambda)$ and every $u \in V^{p,\mu}(M_1 | \R^n)$ satisfying
\begin{equation*}
\mathcal{E}^\mu(u, \varphi) = (f, \varphi) \quad \text{for any}~ \varphi \in H^{p,\mu}_{M_1}(\R^n),
\end{equation*}
we have $u \in C^\alpha$ at $0$ and
\begin{equation*}
|u(x)-u(0)| \leq C\left( \|u\|_{L^\infty(\R^n)} + \|f\|_{L^{q/(p\bar{s})}(M_{15/16})} \right) d(x,0)^\alpha
\end{equation*}
for all $x \in M_1$.
\end{theorem}
\begin{proof}
We assume that $2\|u\|_{L^\infty(\R^n)} + \kappa^{-1} \|f\|_{L^{q/(p\bar{s})}(M_{15/16})} \leq 1$ for some $\kappa > 0$ which will be chosen later. It is enough to construct sequences $\lbrace a_k \rbrace_{k=0}^\infty$ and $\lbrace b_k \rbrace_{k=0}^\infty$ such that $a_k \leq u \leq b_k$ in $M_{1/4^k}$ and $b_k - a_k = 4^{-\alpha k}$ for some $\alpha > 0$. For $k=0$, we set $a_0 = -1/2$ and $b_0 = 1/2$. Assume that we have constructed such sequences up to $k$ and let us choose $a_{k+1}$ and $b_{k+1}$.

We assume
\begin{equation} \label{eq:half_measure}
|\lbrace u \geq (b_k+a_k)/2 \rbrace \cap M_{\frac{1}{2}4^{-k}}| \geq |M_{\frac{1}{2}4^{-k}}|/2,
\end{equation}
and then prove that we can choose $a_{k+1}$ and $b_{k+1}$. If \eqref{eq:half_measure} does not hold, then we can consider $-u$ instead of $u$. Let $\Psi$ be the diffeomorphism given by \eqref{eq:Psi} with $\lambda = 4^{-k}$ and define
\begin{equation*}
v(x) = \frac{u(\Psi(x))-a_k}{(b_k-a_k)/2} \quad\text{and}\quad g(x) = \frac{\lambda^{ps_{\max}} f(\Psi(x))}{(b_k-a_k)/2}.
\end{equation*}
Then, $v \geq 0$ in $M_1$ and $\mathcal{E}^\mu_{M_1}(v,\varphi) = (g, \varphi)$ for every $\varphi \in H^\mu_{M_1}(\R^n)$. Moreover, it is easy to see that $v \geq 2(1-4^{\alpha j}) $ in $M_{4^j}$ for every $j \geq 0$ by induction hypothesis. By applying \Cref{thm:weak_Harnack}, we obtain
\begin{equation} \label{eq:v_WHI}
\begin{split}
&\left( \fint_{M_{1/2}} v^{p_0}(x) \,\mathrm{d}x \right)^{1/p_0} \\
&\leq C\inf_{M_{1/4}} v + C \sup_{x \in M_{15/16}} \left( \int_{\R^n\setminus M_1} v^-(y)^{p-1} \,\mu(x, \d y) \right)^{1/(p-1)} + \|g\|_{L^{q/(p\bar{s})}(M_{15/16})}.
\end{split}
\end{equation}
By taking $\alpha < ps_0 \frac{q-n}{q}$, we have
\begin{equation} \label{eq:g}
\|g\|_{L^{q/(p\bar{s})}(M_{15/16})} = 2\cdot 4^{(\alpha-ps_{\max}\frac{q-n}{q})k} \|f\|_{L^{q/(p\bar{s})}(M_{4^{-k} \cdot 15/16})} \leq 2\kappa.
\end{equation}
For $x \in M_{15/16}$ and for each $j \geq 1$, we have $M_{4^j} \setminus M_{4^{j+1}} \subset \R^n \setminus M_{4^j}(x)$. Hence, by \eqref{assmu1}
\begin{equation} \label{eq:tail}
\begin{split}
\int_{\R^n \setminus M_1} v^-(y)^{p-1} \,\mu(x, \d y)
&\leq \sum_{j=1}^\infty \int_{M_{4^j} \setminus M_{4^{j+1}}} (2(4^{\alpha j}-1))^{p-1} \,\mu(x, \d y) \\
&\leq \sum_{j=1}^\infty (2(4^{\alpha j}-1))^{p-1} \mu(x, \R^n \setminus M_{4^j}(x)) \\
&\leq \sum_{j=1}^l \Lambda(2(4^{\alpha j}-1))^{p-1} 4^{-ps_0 j} + 2^{p-1} \Lambda \sum_{j=l+1}^\infty 4^{(\alpha (p-1)-ps_0)j}.
\end{split}
\end{equation}
If we assume that $\alpha < \frac{ps_0}{2(p-1)}$, then we can make the last term in \eqref{eq:tail} as small as we want by taking $l = l(p,s_0)$ sufficiently large. Since the first term in \eqref{eq:tail} converges to 0 as $\alpha \to 0$, we have
\begin{equation} \label{eq:tail2}
C \sup_{x \in M_{15/16}} \left( \int_{\R^n\setminus M_1} v^-(y)^{p-1} \,\mu(x, \d y) \right)^{1/(p-1)} \leq \kappa
\end{equation}
by assuming further that $\alpha = \alpha(n, p, p_{\star}, q, s_0, \Lambda)$ is sufficiently small.

On the other hand, it follows from \eqref{eq:half_measure} that
\begin{equation} \label{eq:L_p0}
\left( \fint_{M_{1/2}} v^{p_0}(x) \,\mathrm{d}x \right)^{1/p_0} \geq \left( \frac{1}{|M_{1/2}|} \int_{M_{1/2} \cap \lbrace v \geq 1 \rbrace} v^{p_0}(x) \,\d x \right)^{1/p_0} \geq 2^{-1/p_0}.
\end{equation}
Combining \eqref{eq:v_WHI}, \eqref{eq:g}, \eqref{eq:tail2}, and \eqref{eq:L_p0}, and choosing $\kappa > 0$ sufficiently small, we arrive at $\inf_{M_{1/4}} v \geq \kappa_0$ for some $\kappa_0$. We take $a_{k+1} = a_k + \kappa_0 (b_k-a_k)/2$ and $b_{k+1}=b_k$, and make $\alpha$ and $\kappa_0$ small so that $1-\kappa_0/2 = 4^{-\alpha}$. Then $a_{k+1} \leq u \leq b_{k+1}$ in $M_{4^{-(k+1)}}$ and $b_{k+1}-a_{k-1} = 4^{-\alpha(k+1)}$, which finishes the proof.
\end{proof}

\begin{appendix}

\section{Algebraic inequalities} \label{sec:ineq}

In this section we prove \Cref{lem:alg_ineq} using the series of lemmas below.

\begin{lemma} \label{lem:ineq1}
Let $a, b > 0$ and $t > p-1 > 0$. Then,
\begin{equation*}
|b-a|^{p-2}(b-a)(a^{-t} - b^{-t}) \geq t\left( \frac{p}{t-p+1} \right)^p \left| a^{\frac{-t+p-1}{p}} - b^{\frac{-t+p-1}{p}} \right|^p.
\end{equation*}
\end{lemma}

\begin{proof}
We may assume that $b > a$. Let $f(x) = -x^{\frac{-t+p-1}{p}}$ and $g(x) = -x^{-t}$, then by using Jensen's inequality we have
\begin{equation*}
\begin{split}
\left| \frac{f(b)-f(a)}{b-a} \right|^p
&= \left| \fint_a^b f'(x) \,\d x \right|^p \leq \fint_a^b (f'(x))^p \,\d x \\
&= \frac{1}{t} \left( \frac{t-p+1}{p} \right)^p \fint_a^b g'(x) \,\d x = \frac{1}{t} \left( \frac{t-p+1}{p} \right)^p \frac{g(b)-g(a)}{b-a},
\end{split}
\end{equation*}
which proves the lemma.
\end{proof}

\begin{lemma} \label{lem:ineq2}
Let $a, b > 0$ and $t > p-1 > 0$. Then,
\begin{equation*}
|b-a|^{p-1} \min\lbrace a^{-t}, b^{-t} \rbrace \leq \left( \frac{p}{t-p+1} \right)^{p-1} \left| a^{\frac{-t+p-1}{p}} - b^{\frac{-t+p-1}{p}} \right|^{p-1} \min\left\lbrace a^{\frac{-t+p-1}{p}}, b^{\frac{-t+p-1}{p}} \right\rbrace.
\end{equation*}
\end{lemma}

\begin{proof}
We may assume that $b > a$. Let $f(x) = -x^{\frac{-t+p-1}{p}}$, then
\begin{equation*}
\begin{split}
\left| \frac{f(b)-f(a)}{b-a} \right|^{p-1}
&= \left| \fint_a^b f'(x) \,\d x \right|^{p-1} = \left( \frac{t-p+1}{p} \right)^{p-1} \left( \fint_a^b x^{\frac{-t-1}{p}} \,\d x \right)^{p-1} \\
&\geq \left( \frac{t-p+1}{p} \right)^{p-1} \left( \fint_a^b b^{\frac{-t-1}{p}} \,\d x \right)^{p-1} = \left( \frac{t-p+1}{p} \right)^{p-1} \frac{b^{-t}}{b^{\frac{-t+p-1}{p}}},
\end{split}
\end{equation*}
which proves the lemma.
\end{proof}

\begin{lemma} \label{lem:ineq3}
Let $\tau_1, \tau_2 \in [0,1]$ and $p > 1$. Then,
\begin{equation*}
|\tau_1^p - \tau_2^p| \leq p |\tau_1-\tau_2| \max\lbrace \tau_1^{p-1},\tau_2^{p-1} \rbrace.
\end{equation*}
\end{lemma}

\begin{proof}
The desired inequality follows from the convexity of the function $f(\tau) = \tau^p$.
\end{proof}

\begin{lemma} \label{lem:ineq4}
Let $a, b > 0$, $\tau_1, \tau_2 \in [0,1]$, and $t > p-1 > 0$. Then,
\begin{equation} \label{eq:min}
\begin{split}
&\min\lbrace \tau_1^p, \tau_2^p \rbrace \left| a^{\frac{-t+p-1}{p}} - b^{\frac{-t+p-1}{p}} \right|^p \\
&\geq 2^{1-p} \left| \tau_1 a^{\frac{-t+p-1}{p}} - \tau_2 b^{\frac{-t+p-1}{p}} \right|^p - |\tau_1-\tau_2|^p \max \lbrace a^{-t+p-1}, b^{-t+p-1} \rbrace
\end{split}
\end{equation}
and
\begin{equation} \label{eq:max}
\begin{split}
&\max\lbrace \tau_1^p, \tau_2^p \rbrace \left| a^{\frac{-t+p-1}{p}} - b^{\frac{-t+p-1}{p}} \right|^p \\
&\leq 2^{p-1} \left| \tau_1 a^{\frac{-t+p-1}{p}} - \tau_2 b^{\frac{-t+p-1}{p}} \right|^p + 2^{p-1} |\tau_1-\tau_2|^p \max \lbrace a^{-t+p-1}, b^{-t+p-1} \rbrace.
\end{split}
\end{equation}
\end{lemma}

\begin{proof}
For \eqref{eq:min}, we assume that $\tau_1 \geq \tau_2$. Then, we obtain from
\begin{equation*}
\tau_1 a^{\frac{-t+p-1}{p}} - \tau_2 b^{\frac{-t+p-1}{p}} = \tau_2 \left( a^{\frac{-t+p-1}{p}} - b^{\frac{-t+p-1}{p}} \right) + (\tau_1-\tau_2) a^{\frac{-t+p-1}{p}}
\end{equation*}
that
\begin{equation*}
\begin{split}
2^{1-p} \left| \tau_1 a^{\frac{-t+p-1}{p}} - \tau_2 b^{\frac{-t+p-1}{p}} \right|^p 
&\leq \tau_2^p \left|a^{\frac{-t+p-1}{p}} - b^{\frac{-t+p-1}{p}}\right|^p + |\tau_1-\tau_2|^p a^{-t+p-1} \\
&\leq \tau_2^p \left|a^{\frac{-t+p-1}{p}} - b^{\frac{-t+p-1}{p}}\right|^p + |\tau_1-\tau_2|^p \max \lbrace a^{-t+p-1}, b^{-t+p-1} \rbrace,
\end{split}
\end{equation*}
from which \eqref{eq:min} follows. The other case $\tau_1 < \tau_2$ can be proved in the same way.

For \eqref{eq:max}, we assume that $\tau_1 \geq \tau_2$. Then
\begin{equation*}
\begin{split}
\max\lbrace \tau_1^p, \tau_2^p \rbrace \left| a^{\frac{-t+p-1}{p}} - b^{\frac{-t+p-1}{p}} \right|^p
&= \left| \left( \tau_1 a^{\frac{-t+p-1}{p}} - \tau_2 b^{\frac{-t+p-1}{p}} \right) + (\tau_2-\tau_1) b^{\frac{-t+p-1}{p}} \right|^p \\
&\leq 2^{p-1} \left| \tau_1 a^{\frac{-t+p-1}{p}} - \tau_2 b^{\frac{-t+p-1}{p}} \right|^p + 2^{p-1} |\tau_2-\tau_1|^p b^{-t+p-1}.
\end{split}
\end{equation*}
The proof for the case $\tau_1 < \tau_2$ is the same.
\end{proof}

\begin{proof} [Proof of \Cref{lem:alg_ineq}]
We may assume that $b > a$. We begin with the equality
\begin{equation} \label{eq:AB}
\begin{split}
&|b-a|^{p-2}(b-a)(\tau_1^p a^{-t} - \tau_2^p b^{-t}) \\
&= (b-a)^{p-1}(a^{-t} - b^{-t})\tau_1^p + (b-a)^{p-1} b^{-t}(\tau_1^p-\tau_2^p) =: A + B.
\end{split}
\end{equation}
By \Cref{lem:ineq1} and \eqref{eq:min}, we have
\begin{equation} \label{eq:A}
\begin{split}
A
&\geq t\left( \frac{p}{t-p+1} \right)^p \left| a^{\frac{-t+p-1}{p}} - b^{\frac{-t+p-1}{p}} \right|^p \min\lbrace \tau_1^p, \tau_2^p \rbrace \\
&\geq t\left( \frac{p}{t-p+1} \right)^p \left( 2^{1-p} \left| \tau_1 a^{\frac{-t+p-1}{p}} - \tau_2 b^{\frac{-t+p-1}{p}} \right|^p - |\tau_1-\tau_2|^p \max \lbrace a^{-t+p-1}, b^{-t+p-1} \rbrace \right).
\end{split}
\end{equation}
For $B$, we use \Cref{lem:ineq2}, \Cref{lem:ineq3}, and Young's inequality to obtain
\begin{equation*}
\begin{split}
B
&\geq - p \left( \frac{p}{t-p+1} \right)^{p-1} \left| a^{\frac{-t+p-1}{p}} - b^{\frac{-t+p-1}{p}} \right|^{p-1} b^{\frac{-t+p-1}{p}} |\tau_1-\tau_2| \max\lbrace \tau_1^{p-1}, \tau_2^{p-1} \rbrace \\
&\geq - (p-1) \left( \frac{p}{t-p+1} \right)^p \varepsilon^{p/(p-1)} \left| a^{\frac{-t+p-1}{p}} - b^{\frac{-t+p-1}{p}} \right|^p \max\lbrace \tau_1^p, \tau_2^p \rbrace - \frac{1}{\varepsilon^p} b^{-t+p-1} |\tau_1-\tau_2|^p
\end{split}
\end{equation*}
for any $\varepsilon > 0$. Using \eqref{eq:max}, we have
\begin{equation} \label{eq:B}
\begin{split}
B
&\geq - 2^{p-1} (p-1) \left( \frac{p}{t-p+1} \right)^p \varepsilon^{p/(p-1)} \left| \tau_1 a^{\frac{-t+p-1}{p}} - \tau_2 b^{\frac{-t+p-1}{p}} \right|^p \\
&\quad - \left( 2^{p-1} (p-1) \left( \frac{p}{t-p+1} \right)^p \varepsilon^{p/(p-1)} + \frac{1}{\varepsilon^p} \right) |\tau_1-\tau_2|^p \max \lbrace a^{-t+p-1}, b^{-t+p-1} \rbrace.
\end{split}
\end{equation}
Combining \eqref{eq:AB}, \eqref{eq:A}, and \eqref{eq:B}, and then taking $\varepsilon$ so that $2^{p-1} \varepsilon^{p/(p-1)} = 2^{1-p}$, we arrive at
\begin{equation*}
\begin{split}
&|b-a|^{p-2}(b-a)(\tau_1^p a^{-t} - \tau_2^p b^{-t}) \\
&\geq c_1 \left| \tau_1 a^{\frac{-t+p-1}{p}} - \tau_2 b^{\frac{-t+p-1}{p}} \right|^p - c_2 |\tau_1-\tau_2|^p \left( a^{-t+p-1} + b^{-t+p-1} \right),
\end{split}
\end{equation*}
where
\begin{equation*}
c_1 = \frac{2^{1-p} p^p}{(t-p+1)^{p-1}} \quad\text{and}\quad c_2 = \left( t+2^{1-p}(p-1) \right) \left( \frac{p}{t-p+1} \right)^p + 2^{2(p-1)^2}.
\end{equation*}
Note that $c_1$ and $c_2$ are bounded when $t$ is bounded away from $p-1$.
\end{proof}

\section{Anisotropic dyadic rectangles}\label{sec:anisorect}
Let us briefly sketch the construction of anisotropic ``dyadic" rectangles. These objects can be used to prove the lower bound in $L^p$ for the sharp maximal function ${\bf M}^{\sharp}u$.

We construct anisotropic dyadic rectangles having the following properties:
\begin{enumerate}[(i)]
\item
For each integer $k \in \mathbb{Z}$, a countable collection $\lbrace Q_{k, \alpha} \rbrace_\alpha$ covers the whole space $\mathbb{R}^n$.
\item
Each $Q_k$ ($= Q_{k, \alpha}$ for some $\alpha$) has an interior of the form $\mathrm{Int}(Q_k) = M_{2^{-k}}(x)$. We call $Q_k$ an {\it anisotropic dyadic rectangle of generation $k$}.
\item
Every $Q_{k, \alpha}$ is contained in $Q_{k-1, \beta}$ for some $\beta$. We call $Q_{k-1, \beta}$ a {\it predecessor of $Q_{k, \alpha}$}.
\item
If $Q_{k, 0}, \dots, Q_{k, 2^n}$ are $2^n+1$ different anisotropic dyadic rectangles of generation $k$, then $\cap_{i=0}^{2^n} Q_{k, i} = \emptyset$.
\item
If $2^n+1$ different anisotropic dyadic rectangles $Q_{k_0}, \dots, Q_{k_{2^n}}$, $k_0 \leq \dots \leq k_{2^n}$, have a non-empty intersection, then $Q_j \subset Q_i$ for some $0 \leq i < j \leq 2^n$.
\end{enumerate}

\begin{remark} {\ }
\begin{enumerate}
\item
A predecessor may not be unique.
\item
$2^n$ different anisotropic dyadic rectangles from the same generation may have a non-empty intersection.
\end{enumerate}
\end{remark}

Such a family of anisotropic dyadic rectangles can be easily constructed. Since the sets are rectangles, it is sufficient to exemplify the construction in one dimension. Let $Q_0 = [0,1)$. Then, a countable collection $\lbrace Q_0+z \rbrace_{z \in \mathbb{Z}}$ constitutes the zeroth generation. Let $N=\lfloor 2^{s_{\max}/s_1} \rfloor$. In order to construct the first generation, we take a disjoint family of (left-closed and right-opened) $N$ intervals in $Q_0$ starting from 0 with length $2^{-s_{\max}/s_1}$ such that the following interval starts at the endpoint of the previous interval. If the right-endpoint of the last interval is 1, then these intervals constitute the first generation and there is nothing to do. Thus, we assume from now on that $2^{s_{\max}/s_1} \notin \mathbb{Z}$. In this case, we add an interval $[1-2^{-s_{\max}/s_i}, 1)$ so that
\begin{equation*}
Q_0 = \left( \bigcup_{i=0}^{N-1} Q_{1, i} \right) \cup Q_{1, N},
\end{equation*}
where
\begin{equation*}
Q_{1, i} = [i 2^{-s_{\max}/s_1}, (i+1) 2^{-s_{\max}/s_1}) \quad\text{for}~ i=0, \dots, N-1, \quad Q_{1, N} = [1-2^{-s_{\max}/s_i}, 1),
\end{equation*}
and
$N = \lfloor 2^{s_{\max}/s_1} \rfloor$. Then, the collection $\lbrace Q_{1, i}+z \rbrace_{0 \leq i \leq N, z \in \mathbb{Z}}$ forms the first generation of intervals satisfying (i)-(iv).
\begin{figure}[htb]
\includegraphics[width=0.6\textwidth]{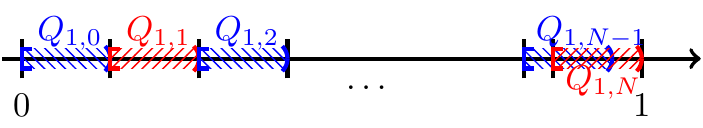}
\caption{This figure shows the construction of the family $Q_{1,i}$.}
\end{figure}

We continue to construct the intervals of generation 2 that fill in $Q_{1, i}$ for each $0 \leq i \leq N-2$. However, we have to be careful in filling in $Q_{1, N-1}$ and $Q_{1, N}$ since $Q_{1, N-1} \cap Q_{1,N} \neq \emptyset$. Suppose that we filled in $Q_{1, N-1}$ and $Q_{1, N}$ as above, i.e., 
\begin{equation*}
Q_{1, N-1} = \left( \bigcup_{i=0}^{N-1} Q_{2, i} \right) \cup Q_{2, N} \quad \text{and} \quad Q_{1, N} = \left( \bigcup_{i=0}^{N-1} \tilde{Q}_{2, i} \right) \cup \tilde{Q}_{2, N}
\end{equation*}
for some intervals $Q_{2, i}$ and $\tilde{Q}_{2, i}$, $0 \leq i \leq N$, of length $4^{-s_{\max}/s_1}$. Let $K$ be the smallest integer such that $\overline{Q_{2, K}} \cap Q_{1, N} \neq \emptyset$. Then, we have
\begin{equation*}
Q_{1, N-1} \cup Q_{1, N} = \left( \bigcup_{i=0}^K Q_{2,i} \right) \cup \left( \bigcup_{i=0}^{N-1} \tilde{Q}_{2, i} \right) \cup \tilde{Q}_{2, N}
\end{equation*}
and at most two different intervals among $\lbrace Q_{2, 0}, \dots, Q_{2, K}, \tilde{Q}_{2,0}, \dots, \tilde{Q}_{2, N} \rbrace$ can intersect. Therefore, these intervals constitute the second generation satisfying (i)-(iv).
 \begin{figure}[htb]
  \includegraphics[width=0.8\textwidth]{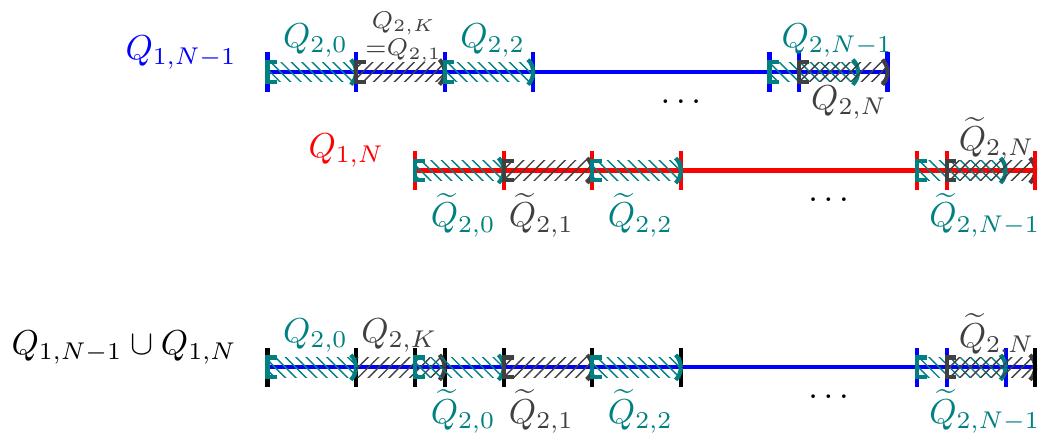}
   \caption{This figure shows the construction of generation 2.}
 \end{figure} \\
In this way, we construct intervals of generation $k$ for all $k \geq 0$.

Let us now construct intervals of generation $k < 0$. It is easy to observe that a collection $\lbrace Q_{-1}+ Nz\rbrace_{z \in \mathbb{Z}}$ of intervals of generation $-1$ satisfies (i)-(iv), where
\begin{equation*}
Q_{-1} = [0, 2^{s_{\max}/s_1}).
\end{equation*}
For the generation of $-2$, let $K$ be the largest integer such that $Q_{-1}+NK \subset Q_{-2}$, where $Q_{-2} = [0, 4^{s_{\max}/s_1})$. Then, the intervals $Q_{-2}+N(K+1)z$, $z \in \mathbb{Z}$, form the generation of $-2$, which satisfying (i)-(iv). We continue this process to construct intervals of all generations $k < 0$.

We show that the intervals constructed in this way satisfy the property (v) as well. Suppose that three different intervals $Q_{k_0}, Q_{k_1}$, and $Q_{k_2}$, $k_0 \leq k_1 \leq k_2$, have a non-empty intersection. If $k_1 = k_2$, then $Q_{k_1} \subset Q_{k_0}$ or $Q_{k_2} \subset Q_{k_0}$. If $k_1 < k_2$, then either $Q_{k_2} \subset Q_{k_1}$ or not. In the former case, we are done. In the latter case, $Q_{k_2} \subset \tilde{Q}_{k_1}$ for some $\tilde{Q}_{k_1} \neq Q_{k_1}$, which reduces to the case $k_1=k_2$.

\section{Sharp maximal function theorem} \label{sec:sharp_maximal}
In this section we prove \Cref{thm:sharp_maximal} by using the anisotropic dyadic rectangles. For $u \in L^1_{\mathrm{loc}}(\mathbb{R}^n)$, we define a dyadic maximal function ${\bf M}_du$ by
\begin{equation*}
{\bf M}_du(x) = \sup_{x \in Q} \fint_{Q} |u(y)| \,\mathrm{d}y,
\end{equation*}
where the supremum is taken over all anisotropic dyadic rectangles $Q$. Since
\begin{equation} \label{eq:MdM}
{\bf M}_d u \leq {\bf M}u,
\end{equation}
\Cref{thm:maximal} also holds for the dyadic maximal function ${\bf M}_du$. We first prove a good-lambda estimate using the dyadic maximal function. See \cite[Theorem 3.4.4]{GrafakosMF}.

\begin{theorem} \label{thm:good_lambda}
Let $s_1, \dots, s_n \in [s_0, 1)$ be given for some $s_0 \in (0,1)$. There exists a constant $C = C(n, s_0) > 0$ such that 
\begin{equation*}
|\lbrace x \in \mathbb{R}^n: {\bf M}_d u(x) > 2\lambda, {\bf M}^{\sharp}u(x) \leq \gamma \lambda \rbrace| \leq C\gamma |\lbrace x \in \mathbb{R}^n: {\bf M}_d u(x) > \lambda \rbrace|
\end{equation*}
for all $\gamma>0$, $\lambda > 0$, and $u \in L^1_{\mathrm{loc}}(\mathbb{R}^n)$.
\end{theorem}

\begin{proof}
Let $\Omega_\lambda = \lbrace x \in \mathbb{R}^n: M_d(f)(x) > \lambda \rbrace$. We may assume that $|\Omega_\lambda|<+\infty$ since otherwise there is nothing to prove. For each $x \in \Omega_\lambda$, we find a maximal anisotropic dyadic rectangle $Q^x$ such that
\begin{equation} \label{eq:Qx}
x \in Q^x \subset \Omega_\lambda \quad\text{and}\quad \fint_{Q^x} |f| > \lambda.
\end{equation}
There are at most $2^n$ different maximal anisotropic dyadic rectangles of the same generation satisfying \eqref{eq:Qx}, but we can still choose anyone of them. Let $Q_j$ be the collection of all such rectangles $Q^x$ for all $x \in \Omega_{\lambda}$. Then, we have $\Omega_\lambda = \cup_j Q_j$. Note that different rectangles $Q_j$ may have an intersection, but the intersection is contained in at most $2^n$ different maximal rectangles of the same generation. This is a consequence of the properties (iv) and (v) of anisotropic dyadic rectangles. Hence,
\begin{equation*}
\sum_{j} |Q_j| \leq 2^n |\Omega_\lambda|.
\end{equation*}
Therefore, the desired result follows once we have
\begin{equation} \label{eq:Q_j}
|\lbrace x \in Q_j: {\bf M}_d u(x) > 2\lambda, {\bf M}^{\sharp} u(x) \leq \gamma \lambda \rbrace| \leq C\gamma |Q_j|
\end{equation}
for some $C = C(n, s_0)$. Indeed, one can prove \eqref{eq:Q_j} by following the second paragraph of the proof of \cite[Theorem 3.4.4]{GrafakosMF}, using \Cref{thm:maximal} for ${\bf M}_d$, and replacing \cite[Equation (3.4.8)]{GrafakosMF} by
\begin{equation*}
\frac{1}{\lambda} \int_{Q_j} |u(y) - (u)_{Q_j'}| \,\mathrm{d}y \leq \frac{2^{ns_{\max}/\bar{s}}}{\lambda} \frac{|Q_j|}{|Q_j'|} \int_{Q_j'} |u(y) - (u)_{Q_j'}| \,\mathrm{d}y \leq \frac{2^{n/s_0}}{\lambda} |Q_j| {\bf M}^{\sharp}u(\xi_j)
\end{equation*}
for all $\xi_j \in Q_j$, where $Q_j'$ is anyone of predecessors of $Q_j$.
\end{proof}

\begin{theorem} \label{thm:Md}
Let $s_1, \dots, s_n \in [s_0, 1)$ be given for some $s_0 \in (0,1)$, and let $0<p_0 \leq p<\infty$. Then, there is a constant $C = C(n, p, s_0) > 0$ such that for all functions $u \in L^1_{\mathrm{loc}}(\mathbb{R}^n)$ with ${\bf M}_du \in L^{p_0}(\mathbb{R}^n)$ we have
\begin{equation*}
\|{\bf M}_du\|_{L^p(\mathbb{R}^n)} \leq C\|{\bf M}^{\sharp}u \|_{L^p(\mathbb{R}^n)}.
\end{equation*}
\end{theorem}

\Cref{thm:Md} can be proved in the same way as in the proof of \cite[Theorem 3.4.5]{GrafakosMF} except that we use \Cref{thm:good_lambda} instead of \cite[Theorem 3.4.4]{GrafakosMF}. Finally, we combine the inequality
\begin{equation*}
\|u\|_{L^p(\mathbb{R}^n)} \leq \|{\bf M}_du\|_{L^p(\mathbb{R}^n)},
\end{equation*}
which comes from the Lebesgue differentiation theorem, and \Cref{thm:Md} to conclude \Cref{thm:sharp_maximal}. See \cite[Corollary 3.4.6]{GrafakosMF}.

\section{Pointwise convergence of the fractional orthotropic \texorpdfstring{$p$}{p}-Laplacian}

This section provides the proof of pointwise convergence of the fractional orthotropic $p$-Laplacian as $s \nearrow 1$.

\begin{proposition} \label{prop:convergence}
Let $u \in C^2(\mathbb{R}^n) \cap L^{\infty}(\mathbb{R}^n)$ and $x \in \mathbb{R}^n$ be such that $\partial_iu(x) \neq 0$ for all $i = 1, \dots, n$. Let $s_i = s$ for all $i=1, \dots, n$. Let $L$ be the operator in \eqref{def:nonlocaloperator} with $\mu = \ma$ and $A^{p}_{\mathrm{loc}}$ be as in \eqref{eq:Aploc}. Then, $Lu(x) \to A^{p}_{\mathrm{loc}}u(x)$ as $s \nearrow 1$ up to a constant.
\end{proposition}

\begin{proof}
Let us fix a point $x \in \mathbb{R}^n$ with $\partial_iu(x) \neq 0$. For each $i = 1, \dots, n$, let us define $u_i : \mathbb{R} \to \mathbb{R}$ by $u_i(x_i) = u(x_1, \dots, x_i, \dots, x_n)$ as a function of one variable. Then $u_i \in C^2(\mathbb{R}) \cap L^{\infty}(\mathbb{R})$ and $u_i'(x_i) \neq 0$. We write
\begin{equation*}
\begin{split}
Lu(x) = \sum_{i=1}^{n} s(1-s) \int_{\mathbb{R}} \frac{|u_i(y_i) - u(x_i)|^{p-2} (u_i(y_i) - u_i(x_i))}{|x_i-y_i|^{1+sp}} \,\d y_i = - \sum_{i=1}^n (-\partial^2)^{s}_{p} u_i(x_i),
\end{split}
\end{equation*}
which is the sum of one-dimensional fractional $p$-Laplacians. By \cite[Theorem 2.8]{BucSqu21}, we have
\begin{equation*}
-(-\partial^2)^{s}_{p} u_i(x_i) \to \frac{\d}{\d x_i} \left( \left| \frac{\d u_i}{\d x_i}(x_i) \right|^{p-2} \frac{\d u_i}{\d x_i}(x_i) \right)
\end{equation*}
as $s \nearrow 1$, for each $i=1, \dots, n$, up to a constant depending on $p$ only. Consequently, by summing up $Lu(x) \to A^{p}_{\mathrm{loc}}u(x)$ as $s \nearrow 1$.
\end{proof}

\end{appendix}

\subsection*{Conflict of Interest} 
Authors state no conflict of interest


\begin{thebibliography}{10}

\bibitem{Kin20}
A.~Banerjee, P.~Garain, and J.~Kinnunen.
\newblock Some local properties of subsolution and supersolutions for a doubly
  nonlinear nonlocal $p$-{L}aplace equation.
\newblock {\em Annali di Matematica Pura ed Applicata} (2021), 1--35.

\bibitem{Bar09}
M.~T. Barlow, R.~F. Bass, Z.-Q. Chen, and M.~Kassmann.
\newblock Non-local {D}irichlet forms and symmetric jump processes.
\newblock {\em Trans. Amer. Math. Soc.}, \textbf{361} (2009), no. 4, 1963--1999.

\bibitem{PalaPseudo}
P.~Baroni, A.~Di~Castro, and G.~Palatucci.
\newblock Intrinsic geometry and {D}e {G}iorgi classes for certain anisotropic
  problems.
\newblock {\em Discrete Contin. Dyn. Syst. Ser. S}, \textbf{10} (2017), no. 4, 647--659.

\bibitem{BaCh10}
R.~F. Bass and Z.-Q. Chen.
\newblock Regularity of harmonic functions for a class of singular stable-like
  processes.
\newblock {\em Math. Z.}, \textbf{266} (2010), no. 3, 489--503.

\bibitem{BLH}
R.~F. Bass and D.~A. Levin.
\newblock Harnack inequalities for jump processes.
\newblock {\em Potential Anal.}, \textbf{17} (2002), no. 4, 375--388.

\bibitem{BEKA04}
M.~Belloni and B.~Kawohl.
\newblock The pseudo-{$p$}-{L}aplace eigenvalue problem and viscosity solutions
  as {$p\to\infty$}.
\newblock {\em ESAIM Control Optim. Calc. Var.}, \textbf{10} (2004), no.1, 28--52.

\bibitem{BoSz07}
K.~Bogdan and P.~Sztonyk.
\newblock Estimates of the potential kernel and {H}arnack's inequality for the
  anisotropic fractional {L}aplacian.
\newblock {\em Studia Math.}, \textbf{181} (2007), no. 2, 101--123.

\bibitem{BB20}
P.~Bousquet and L.~Brasco.
\newblock Lipschitz regularity for orthotropic functionals with nonstandard
  growth conditions.
\newblock {\em Rev. Mat. Iberoam.}, \textbf{36} (2020), no. 7, 1989--2032.

\bibitem{BBLV18}
P.~Bousquet, L.~Brasco, C.~Leone, and A.~Verde.
\newblock On the {L}ipschitz character of orthotropic {$p$}-harmonic functions.
\newblock {\em Calc. Var. Partial Differential Equations}, \textbf{57} (2018), no. 3, Paper No. 88,
  33.

\bibitem{BrLiSc18}
L.~Brasco, E.~Lindgren, and A.~Schikorra.
\newblock Higher {H}\"{o}lder regularity for the fractional {$p$}-{L}aplacian
  in the superquadratic case.
\newblock {\em Adv. Math.}, \textbf{338} (2018), 782--846.

\bibitem{BucSqu21}
C.~Bucur and M.~Squassina.
\newblock An asymptotic expansion for the fractional $p$-{L}aplacian and for
  gradient-dependent nonlocal operators.
\newblock {\em Communications in Contemporary Mathematics} (2021), Online Ready.

\bibitem{CAFFVASS}
L.~Caffarelli, C.~H. Chan, and A.~Vasseur.
\newblock Regularity theory for parabolic nonlinear integral operators.
\newblock {\em J. Amer. Math. Soc.}, \textbf{24} (2011), no. 3, 849--869.

\bibitem{Cha20}
J.~Chaker.
\newblock Regularity of solutions to anisotropic nonlocal equations.
\newblock {\em Math. Z.}, \textbf{296} (2020), no. 3-4, 1135--1155.

\bibitem{CK20}
J.~Chaker and M.~Kassmann.
\newblock Nonlocal operators with singular anisotropic kernels.
\newblock {\em Comm. Partial Differential Equations}, \textbf{45} (2020), no. 1, 1--31.

\bibitem{CKW19}
J.~Chaker, M.~Kassmann, and M.~Weidner.
\newblock Robust {H}\"{o}lder estimates for parabolic nonlocal operators.
\newblock {Preprint (2019), https://arxiv.org/abs/1912.09919}.

\bibitem{KUMA}
Z.-Q. Chen and T.~Kumagai.
\newblock Heat kernel estimates for stable-like processes on {$d$}-sets.
\newblock {\em Stochastic Process. Appl.}, \textbf{108} (2003), no. 1, 27--62.

\bibitem{CHENKUMAWANG}
Z.-Q. Chen, T.~Kumagai, and J.~Wang.
\newblock Elliptic {H}arnack inequalities for symmetric non-local {D}irichlet
  forms.
\newblock {\em J. Math. Pures Appl. (9)}, \textbf{125} (2019), 1--42.

\bibitem{Coz17}
M.~Cozzi.
\newblock Regularity results and {H}arnack inequalities for minimizers and
  solutions of nonlocal problems: a unified approach via fractional {D}e
  {G}iorgi classes.
\newblock {\em J. Funct. Anal.}, \textbf{272} (2017), no. 11, 4762--4837.

\bibitem{DEGIORGI}
E.~De~Giorgi.
\newblock Sulla differenziabilit\`a e l'analiticit\`a delle estremali degli
  integrali multipli regolari.
\newblock {\em Mem. Accad. Sci. Torino. Cl. Sci. Fis. Mat. Nat. (3)}, \textbf{3} (1957), 25--43.

\bibitem{dTGCV20}
F.~del Teso, D.~G\'{o}mez-Castro, and J.~L. V\'{a}zquez.
\newblock Three representations of the fractional {$p$}-{L}aplacian: semigroup,
  extension and {B}alakrishnan formulas.
\newblock {\em Fract. Calc. Appl. Anal.}, \textbf{24} (2021), no. 4, 966--1002.

\bibitem{DCKP14}
A.~Di~Castro, T.~Kuusi, and G.~Palatucci.
\newblock Nonlocal {H}arnack inequalities.
\newblock {\em J. Funct. Anal.}, \textbf{267} (2014), no. 6, 1807--1836.

\bibitem{DCKP16}
A.~Di~Castro, T.~Kuusi, and G.~Palatucci.
\newblock Local behavior of fractional {$p$}-minimizers.
\newblock {\em Ann. Inst. H. Poincar\'{e} Anal. Non Lin\'{e}aire},
  \textbf{33} (2016), no. 5, 1279--1299.

\bibitem{LeitSan20}
E.~dos Santos and R.~Leit{\~a}o.
\newblock On the {H}{\"o}lder regularity for solutions of integro-differential
  equations like the anisotropic fractional {L}aplacian.
\newblock {\em SN Partial Differential Equations and Applications}, \textbf{2} (2021), no. 2, 1--34.

\bibitem{DyKa17}
B.~Dyda and M.~Kassmann.
\newblock Function spaces and extension results for nonlocal {D}irichlet
  problems.
\newblock {\em J. Funct. Anal.}, \textbf{277} (2019), no. 11, 108134, 22.

\bibitem{DK20}
B.~Dyda and M.~Kassmann.
\newblock Regularity estimates for elliptic nonlocal operators.
\newblock {\em Anal. PDE}, \textbf{13} (2020), no. 2, 317--370.

\bibitem{FallReg}
M.~M. Fall.
\newblock Regularity results for nonlocal equations and applications.
\newblock {\em Calc. Var. Partial Differential Equations}, \textbf{59} (2020), no. 5, Paper No. 181,
  53.

\bibitem{KASFELS}
M.~Felsinger and M.~Kassmann.
\newblock Local regularity for parabolic nonlocal operators.
\newblock {\em Comm. Partial Differential Equations}, \textbf{38} (2013), no. 9, 1539--1573.

\bibitem{FKV15}
M.~Felsinger, M.~Kassmann, and P.~Voigt.
\newblock The {D}irichlet problem for nonlocal operators.
\newblock {\em Math. Z.}, \textbf{279} (2015), no. 3-4, 779--809.

\bibitem{FriPen20}
M.~Friesen and P.~Jin.
\newblock On the anisotropic stable {JCIR} process.
\newblock {\em ALEA Lat. Am. J. Probab. Math. Stat.}, \textbf{17} (2020), no. 2, 643--674.

\bibitem{FRIPERU18}
M.~Friesen, P.~Jin, and B.~R\"{u}diger.
\newblock Existence of densities for stochastic differential equations driven
  by {L}\'{e}vy processes with anisotropic jumps.
\newblock {\em Ann. Inst. Henri Poincar\'{e} Probab. Stat.}, \textbf{57} (2021), no. 1, 250--271.

\bibitem{GrafakosMF}
L.~Grafakos.
\newblock {\em Modern {F}ourier analysis}, volume 250 of {\em Graduate Texts in
  Mathematics}.
\newblock Springer, New York, third edition, 2014.

\bibitem{Hein01}
J.~Heinonen.
\newblock {\em Lectures on analysis on metric spaces}.
\newblock Universitext. Springer-Verlag, New York, 2001.

\bibitem{kassmann-apriori}
M.~Kassmann.
\newblock A priori estimates for integro-differential operators with measurable
  kernels.
\newblock {\em Calc. Var. Partial Differential Equations}, \textbf{34} (2009), no. 1, 1--21.

\bibitem{KKK19}
M.~Kassmann, K.-Y. Kim, and T.~Kumagai.
\newblock Heat kernel bounds for nonlocal operators with singular kernels.
\newblock {Preprint (2019), https://arxiv.org/abs/1910.04242},
\newblock to appear in\textit{ J. Math. Pures Appl}.

\bibitem{KassSchwab}
M.~Kassmann and R.~W. Schwab.
\newblock Regularity results for nonlocal parabolic equations.
\newblock {\em Riv. Math. Univ. Parma (N.S.)}, \textbf{5} (2014), no. 1, 183--212.

\bibitem{KuRy17}
T.~Kulczycki and M.~Ryznar.
\newblock Transition density estimates for diagonal systems of {SDE}s driven by
  cylindrical {$\alpha$}-stable processes.
\newblock {\em ALEA Lat. Am. J. Probab. Math. Stat.}, \textbf{15} (2018), no. 2, 1335--1375.

\bibitem{KuRy20}
T.~Kulczycki and M.~Ryznar.
\newblock Semigroup properties of solutions of {SDE}s driven by {L}\'{e}vy
  processes with independent coordinates.
\newblock {\em Stochastic Process. Appl.}, \textbf{130} (2020), no. 12, 7185--7217.

\bibitem{KuMiSi15}
T.~Kuusi, G.~Mingione, and Y.~Sire.
\newblock Nonlocal equations with measure data.
\newblock {\em Comm. Math. Phys.}, \textbf{337} (2015), no. 3, 1317--1368.

\bibitem{ToIrFe15}
T.~Leonori, I.~Peral, A.~Primo, and F.~Soria.
\newblock Basic estimates for solutions of a class of nonlocal elliptic and
  parabolic equations.
\newblock {\em Discrete Contin. Dyn. Syst.}, \textbf{35} (2015), no. 12, 6031--6068.

\bibitem{Lind16}
E.~Lindgren.
\newblock H\"{o}lder estimates for viscosity solutions of equations of
  fractional {$p$}-{L}aplace type.
\newblock {\em NoDEA Nonlinear Differential Equations Appl.}, \textbf{23} (2016), no. 1, Art. 55,
  18.

\bibitem{Lions69}
J.-L. Lions.
\newblock {\em Quelques m\'{e}thodes de r\'{e}solution des probl\`emes aux
  limites non lin\'{e}aires}.
\newblock Dunod; Gauthier-Villars, Paris, 1969.

\bibitem{Rod21}
C.~W. Lo and J.~F. Rodrigues.
\newblock On a class of fractional obstacle type problems related to the
  distributional {R}iesz derivative.
\newblock {Preprint (2021),  https://arxiv.org/abs/2101.06863}.

\bibitem{Ming11}
G.~Mingione.
\newblock Gradient potential estimates.
\newblock {\em J. Eur. Math. Soc. (JEMS)}, \textbf{13} (2011), no. 2, 459--486.

\bibitem{Mosconi}
S.~J.~N. Mosconi.
\newblock Optimal elliptic regularity: a comparison between local and nonlocal
  equations.
\newblock {\em Discrete Contin. Dyn. Syst. Ser. S}, \textbf{11} (2018), no. 3, 547--559.

\bibitem{MOSER}
J.~Moser.
\newblock On {H}arnack's theorem for elliptic differential equations.
\newblock {\em Comm. Pure Appl. Math.}, \textbf{14} (1961), 577--591.

\bibitem{NASH}
J.~Nash.
\newblock Continuity of solutions of parabolic and elliptic equations.
\newblock {\em Amer. J. Math.}, \textbf{80} (1958), 931--954.

\bibitem{NDN20}
A.~D. Nguyen, J.~I. D\'{\i}az, and Q.-H. Nguyen.
\newblock Fractional {S}obolev inequalities revisited: the maximal function
  approach.
\newblock {\em Atti Accad. Naz. Lincei Rend. Lincei Mat. Appl.},
  \textbf{31} (2020), no. 1, 225--236.

\bibitem{NOW20}
S.~Nowak.
\newblock {$H^{s,p}$} regularity theory for a class of nonlocal elliptic
  equations.
\newblock {\em Nonlinear Anal.}, \textbf{195} (2020), 111730, 28.

\bibitem{NOW21}
S.~Nowak.
\newblock Higher {H}\"{o}lder regularity for nonlocal equations with irregular
  kernel.
\newblock {\em Calc. Var. Partial Differential Equations}, \textbf{60} (2021), no. 1, Paper No. 24,
  37.

\bibitem{NOW21c}
S.~Nowak.
\newblock Improved sobolev regularity for linear nonlocal equations with {VMO}
  coefficients.
\newblock {Preprint (2021), https://arxiv.org/abs/2108.02856},
\newblock to appear in \textit{Math. Ann}.

\bibitem{NOW21b}
S.~Nowak.
\newblock Regularity theory for nonlocal equations with {VMO} coefficients.
\newblock {Preprint (2021), https://arxiv.org/abs/2101.11690},
\newblock to appear in \textit{Ann. Inst. H. Poincar\'{e} Anal. Non Lin\'{e}aire}.

\bibitem{ROV16}
X.~Ros-Oton and E.~Valdinoci.
\newblock The {D}irichlet problem for nonlocal operators with singular kernels:
  convex and nonconvex domains.
\newblock {\em Adv. Math.}, \textbf{288} (2016), 732--790.

\bibitem{SilvInd}
L.~Silvestre.
\newblock H\"{o}lder estimates for solutions of integro-differential equations
  like the fractional {L}aplace.
\newblock {\em Indiana Univ. Math. J.}, \textbf{55} (2006), no. 3, 1155--1174.

\bibitem{Str17}
M.~Str\"{o}mqvist.
\newblock Harnack's inequality for parabolic nonlocal equations.
\newblock {\em Ann. Inst. H. Poincar\'{e} Anal. Non Lin\'{e}aire},
  \textbf{36} (2019), no. 6, 1709--1745.

\bibitem{Str18}
M.~Str\"{o}mqvist.
\newblock Local boundedness of solutions to non-local parabolic equations
  modeled on the fractional {$p$}-{L}aplacian.
\newblock {\em J. Differential Equations}, \textbf{266} (2019), no. 12, 7948--7979.

\bibitem{ZHANG15}
L.~Wang and X.~Zhang.
\newblock Harnack inequalities for {SDE}s driven by cylindrical
  {$\alpha$}-stable processes.
\newblock {\em Potential Anal.}, \textbf{42} (2015), no. 3, 657--669.
\end{thebibliography}

\end{document}